\documentclass{mcom-l}

\usepackage{float}
\usepackage{graphicx}
\usepackage{algorithm}
\usepackage{algpseudocode}%
\usepackage{amsmath,amssymb,amsfonts}
\usepackage[hidelinks]{hyperref} 
\usepackage{cleveref} 
\usepackage{xcolor}
\usepackage{booktabs}
\usepackage{mathtools}

\newtheorem{theorem}{Theorem}[section]
\newtheorem{lemma}[theorem]{Lemma}
\newtheorem{proposition}{Proposition}[section]
\theoremstyle{definition}
\newtheorem{definition}[theorem]{Definition}

\newtheorem{corollary}{Corollary}[section]
\theoremstyle{remark}
\newtheorem{remark}[theorem]{Remark}
\numberwithin{equation}{section}

\def\d{\displaystyle}

\def\Diag{\operatorname{Diag}}

\def\LHSC{\operatorname{LHSCB}}
\def\cond{\operatorname{Cond}}
\def\interior{\operatorname{int}}
\def\tr{\operatorname{tr}}
\def\det{\operatorname{det}}
\def\P{\mathcal{P}}
\def\K{\mathbb{K}}
\def\J{\mathbb{J}}
\def\V{\mathbb{V}}
\def\E{\mathbb{E}}
\def\A{\mathcal{A}}

\def\argmin{\operatorname{argmin}}

\def\nal{Newton augmented Lagrangian }
\def\NAL{NAL method }
\def\NALM{NAL method}
\def\NALMM{NAL}
\def\L{\mathcal{L}}

\def\varnewphi{\phi}

\begin{document}

\title[An NAL Method for SCP with Complexity Analysis]{A Newton Augmented Lagrangian Method for Symmetric Cone Programming with Complexity Analysis}


\author{Rui-Jin Zhang}
\address{School of Mathematical Sciences and LPMC, Nankai University, Tianjin, 300071, China}
\email{zhangrj@nankai.edu.cn}
\thanks{The first author was supported by the National Natural Science Foundation of China (No. 1250012017), the Natural Science Foundation of Tianjin, China (No. 24JCQNJC01970), and the Fundamental Research Funds for the Central Universities (No. 050-63253088). The third  author was supported by the National Natural Science Foundation of China (Nos. 12071108 and 12471286). The fourth author was supported by the National Natural Science Foundation of China (Nos. 12021001, 11991021).}

\author{Ruoyu Diao}
\address{State Key Laboratory of Mathematical Sciences, Academy of Mathematics and Systems Science, Chinese Academy of Sciences, Beijing 100190, China, and School of Mathematical Sciences, University of Chinese Academy of Sciences, Beijing 100049, China}
\email{diaoruoyu18@mails.ucas.ac.cn}

\author{Xin-Wei Liu}
\address{Institute of Mathematics, Hebei University of Technology, Tianjin, 300401, China }
\email{mathlxw@hebut.edu.cn}

\author{Yu-Hong Dai}
\address{State Key Laboratory of Mathematical Sciences, Academy of Mathematics and Systems Science, Chinese Academy of Sciences, Beijing 100190, China, and School of Mathematical Sciences, University of Chinese Academy of Sciences, Beijing 100049, China}
\email{dyh@lsec.cc.ac.cn}

\subjclass[2020]{Primary 49M15,90C05,90C22,90C25}

\date{}

\dedicatory{}
\keywords{symmetric cone, augmented Lagrangian method, Newton method, complexity bound}
\begin{abstract}
Symmetric cone programming covers a broad class of convex optimization problems, including linear programming, second-order cone programming, and semidefinite programming. Although the augmented Lagrangian method (ALM) is well-suited for large-scale problems, its subproblems are often not twice continuously differentiable, preventing the direct use of classical Newton methods. To address this issue, we observe that barrier functions used in interior-point methods (IPMs) naturally serve as effective smoothing terms to alleviate such nonsmoothness. By combining the strengths of ALM and IPMs, we construct a novel augmented Lagrangian function and subsequently develop a Newton augmented Lagrangian (NAL) method. By leveraging the self-concordance property of the barrier function, the proposed method is shown to achieve an $\mathcal{O}(1/{\epsilon})$ complexity bound.  In addition, a spectral analysis reveals that the condition numbers of the Schur complement matrices arising in the NAL method are of order $\mathcal{O}(1/{\mu})$, which is better than the $\mathcal{O}(1/{\mu^2})$ order of classical IPMs. This improvement is further illustrated by a heatmap of condition numbers. Numerical experiments conducted on standard benchmarks indicate that the NAL method exhibits significant performance improvements compared to several existing methods.
\end{abstract}

\maketitle


\section{Introduction}\label{sec0}
Symmetric cone programming (SCP) is a class of convex optimization problems with affine constraints and symmetric cone constraints on the decision variable. The standard form of SCP problems can be described by the following primal-dual pair:
\begin{equation}\label{cone problem}
	\begin{array}{crllcl}
		&\min & \langle{c}, {x}\rangle & & \min &  -\langle{b}, \lambda \rangle \\
		(\operatorname{P}) &\quad \text { s.t.} & \mathcal{A} {x}={b}, & \quad\quad (\operatorname{D}) \quad & \text { s.t.} & \mathcal{A}^{*} \lambda+{s}={c}, \\
		&& {x} \in \K, & & & {s} \in \K^*,
	\end{array}
\end{equation}
where $\mathcal{A}$ is a linear operator, $\mathcal{A}^{*}$ is its adjoint, and $b$ and $c$ are given data. Throughout this paper, we assume that $\mathcal{A}$ is surjective and that the Slater condition holds for both the primal problem ($\operatorname{P}$) and the dual problem ($\operatorname{D}$).  The set $\K$ is a symmetric cone, defined as a closed, convex cone that is self-dual ($\K=\K^*$) and homogeneous. Notably, any symmetric cone is the cone of squares $\left\{x\circ x \mid x\in \J\right\}$ in some Euclidean Jordan algebra $\E:=(\J, \circ, \langle\cdot,\cdot\rangle)$. In this construction, $\J$ is a finite-dimensional space  equipped with a Jordan product $ \circ: \J \times \J \to \J$ and an inner product $\langle\cdot,\cdot\rangle$ satisfying the following properties for all elements $x,\,y,\,z\in\J$:
\begin{align*}
     x \circ y = y \circ x,\quad
     x^2 \circ (x \circ y) = x \circ (x^2 \circ y),\quad
     \langle x\circ y,z \rangle = \langle y,x\circ z \rangle,
\end{align*}
where $x^2=x\circ x$. Additionally, $\J$ contains an identity element $e$ defined by the property that $x \circ e = x$ for all $x \in \J$. For more details about the  Euclidean Jordan algebra, the reader is referred to \cite{faraut1994analysis,vieira2007jordan}. Symmetric cones encompass several well-known special cases, including the nonnegative orthant for linear programming (LP), the second-order cone for second-order cone programming (SOCP), and the positive semidefinite cone for semidefinite programming (SDP). These cones have found extensive applications in various fields, such as wireless communications \cite{liu2024survey}, control systems \cite{rao1998application}, and machine learning \cite{khemchandani2010learning}, where the development of efficient and reliable optimization methods is essential.

Interior-point methods (IPMs) are among the most efficient methods for SCP, and several widely used software packages include SeDuMi \cite{sturm1999using}, Clarabel \cite{goulart2024clarabel}, and Hypatia \cite{coey2023performance}. For further background on IPMs, see, e.g., \cite{byrd1999interior,chiu2024well,nesterov1994interior,wright1997primal} and the references therein. However, for large-scale SCP problems, the computational cost of forming and factorizing Schur complement matrices (SCMs) in IPMs can become prohibitively expensive. Even when iterative methods are employed to avoid explicitly forming the SCMs, their condition numbers can worsen significantly as the iterates approach the optimal solution, leading to computational challenges that slow down the convergence.

Compared to IPMs, the augmented Lagrangian method (ALM) offers significant computational advantages in addressing the numerical issue mentioned above. For instance, an ALM based on the semismooth Newton method was proposed in \cite{zhao2010newton}, leveraging the sparsity of both the problem and solution to address the SCMs. This method, implemented in SDPNAL$+$ \cite{yang2015sdpnal+}, was particularly effective for large-scale SDP problems and demonstrated outstanding performance in benchmark tests. Extensions to LP and SOCP were discussed in \cite{li2020asymptotically,liang2021inexact}. In recent years, the ALM has also found applications in other problems. We refer to \cite{burer2003nonlinear,chen2017augmented,dai2023augmented,han2014augmented,liang2022qppal,wang2025augmented,yang2013linearized,zhang2022solving} and the references therein for details. All these remarkable works highlight the computational efficiency and potential of the ALM.

Beyond numerical achievements, several studies have established iteration complexity results for the ALM in constrained convex programming. For example, Lan \cite{lan2009convex} provided a detailed complexity analysis for an inexact ALM scheme in which each subproblem was solved by Nesterov’s accelerated gradient method with a fixed penalty parameter. Subsequently, with a specific set of subproblem termination accuracy parameters, Lan and Monteiro \cite{lan2016iteration} derived an $\mathcal{O}(1/(\epsilon_p \epsilon_d^{\frac{3}{4}}))$ complexity bound to obtain an $(\epsilon_p,\epsilon_d)$-KKT solution. Xu \cite{xu2021iteration} established an $\mathcal{O}(1/\epsilon)$ complexity bound to obtain a primal $\epsilon$-solution for general convex optimization problems with equality and inequality constraints. For convex conic programming problems, Lu and Zhou \cite{lu2023iteration} derived an $\mathcal{O}(1/\epsilon^{\frac{7}{4}})$ complexity bound to find an $\epsilon$-KKT solution. 

Despite these achievements, two critical numerical issues arise when applying the ALM to solve SCP problems. First, the augmented Lagrangian function is not twice continuously differentiable, which prevents the use of classical Newton methods for high-precision subproblem solutions. This limitation becomes particularly significant in later iterations, as Rockafellar’s analysis \cite{rockafellar1976augmented,rockafellar1976monotone} demonstrates that solving the subproblems with high precision is required for convergence. Second, selecting appropriate stopping criteria for the subproblems involves a trade-off between computational expense and convergence guarantees. Specifically, excessive precision in early iterations can significantly increase the computational cost, while insufficient precision in later stages may lead to slow convergence or even cause the algorithm to fail. Notably, practical implementation \cite{yang2015sdpnal+} adopts stopping criteria based on Rockafellar's framework. These criteria require determining the minimum of the augmented Lagrangian function, which, however, is unknown in practice.

Motivated by advancements in both the ALM and IPMs, recent developments have combined these two approaches to achieve significant algorithmic improvements. For instance, Wang et al. \cite{wang2010solving} enhanced the classical ALM by incorporating a logarithmic barrier term with a fixed parameter, enabling the resulting subproblems to be efficiently solved via the Newton-CG method. The ADMM-based interior-point method (ABIP) \cite{deng2024enhanced,lin2021admm}, which maintains a fixed coefficient matrix in each iteration, efficiently reduces computational overhead, making it particularly suitable for large-scale problems. Furthermore, recently developed algorithms \cite{liu2020globally,liu2020primal,liu2023novel,zhang2023iprqp,zhang2024iprsdp,zhang2024iprsocp} based on a novel augmented Lagrangian function have exhibited superior numerical behavior, benefiting from the favorable properties of their SCMs compared to classical IPMs. However, despite their promising numerical results, the iteration complexity of these methods \cite{liu2020globally,liu2020primal,liu2023novel,zhang2023iprqp,zhang2024iprsdp,zhang2024iprsocp} has not been fully analyzed. Therefore, further exploration of this novel augmented Lagrangian function is essential to develop more efficient algorithms and strengthen their theoretical foundations.

The main contributions of the Newton augmented Lagrangian (\NALMM) method proposed in this paper can be categorized into algorithmic, theoretical, and numerical aspects:
\begin{itemize}
    \item Algorithmically, the \NAL utilizes a novel augmented Lagrangian function that incorporates a barrier term, similar to the approach described in \cite{wang2010solving}, but with a dynamically updated barrier parameter instead of a fixed one. This dynamic updating improves numerical stability and enables efficient solutions to subproblems through classical Newton methods. Furthermore, the novel augmented Lagrangian function is easier to estimate and control, both theoretically and numerically (see Remark~\ref{inexact-lambda}), compared to traditional formulations.
    \item Theoretically, this paper introduces a novel monotone operator derived from the logarithmic barrier function, which is employed to  prove the global convergence of the \NALM. Additionally, the integration of self-concordance theory from IPMs leads to an easy-to-implement stopping criterion and establishes an $\mathcal{O}(1/{\epsilon})$ complexity bound. In addition, we prove that the SCMs arising in the \NAL possess condition numbers of order $\mathcal{O}(1/{\mu})$, which is better than the $\mathcal{O}(1/{\mu^2})$ order of classical IPMs.
     \item Numerically, the improved conditioning predicted by our theory is confirmed by a heatmap of the condition numbers of the SCMs, clearly demonstrating the advantage of the \NAL over classical IPMs. Additional experiments on standard benchmark problems further show that the \NAL achieves significant performance improvements compared to several existing methods for SCP problems.
\end{itemize}

The rest of this paper is structured as follows. Section~\ref{sec1} introduces the fundamental results of Euclidean Jordan algebra, self-concordance theory, and monotone operators, which are necessary for understanding the subsequent developments. In Section~\ref{sec2}, we propose a novel augmented Lagrangian function and analyze its relevant properties, particularly its self-concordance. Section~\ref{sec3} presents details on the \NAL for solving SCP problems, including convergence analysis and computational procedures. The $\mathcal{O}(1/\epsilon)$ complexity bound of this method and the $\mathcal{O}(1/\mu)$ order of the condition numbers of the SCMs are established in Section~\ref{sec4}, providing theoretical guarantees. Section~\ref{sec5} presents a heatmap and benchmark results highlighting the superior efficiency of the \NAL compared to other solvers. The last section summarizes the main conclusions of this paper. The appendices provide detailed proofs of the theoretical results and extensive numerical details.




\textbf{Notation.} Throughout this paper, we use the following notations.  Let $x^{(k)}$ be the value of $x$  at the $k$-th outer iteration and $x^{(k,j)}$ be the value of $x$ at the $j$-th inner iteration within the $k$-th outer iteration. For a vector $x$, $\text{Diag}(x)$ denotes the diagonal matrix with entries of $x$ on its diagonal. The interior of a cone $\K$  is denoted by $\text{int}\,(\K)$. We use $\nabla_x$ and $\nabla_x^2$ for the gradient and Hessian operators with respect to $x$, respectively, while  $D_x$ denotes the derivative operator. Derivatives with respect to a scalar variable $\mu$ are marked by $ \{\cdot\}'_{\mu}$ or $'$. Following Nesterov's notation \cite{nesterov2018lectures}, the second and third directional derivatives of a function $\phi$ in the directions of $h_1$, $h_2$, and $h_3$ are expressed as $D^2 \phi[h_1, h_2]$ and $D^3 \phi[h_1, h_2, h_3]$, respectively. For any $x\in\J$, let $\{\lambda_i(x)\}$ denote the eigenvalues of $x$ arising from its spectral decomposition in the underlying Euclidean Jordan algebra. In particular, $\lambda_{\max}(x):=\max_i\{\lambda_i(x)\}$ and $\lambda_{\min}(x):=\min_i\{\lambda_i(x)\}$ refer to the largest and smallest eigenvalues of $x$, respectively. Additional notations and symbols are introduced as needed throughout the paper.

\section{Preliminaries}\label{sec1}
This section introduces some fundamental results, including Euclidean Jordan algebra, self-concordance, and monotone operators, which are used in the subsequent sections.

\subsection{Euclidean Jordan algebra}
In any Euclidean Jordan algebra $\E$, each element $x \in \J$ admits a spectral decomposition. That is, there exist orthogonal idempotents $v_1, v_2, \dots, v_r$ and corresponding eigenvalues $\lambda_1(x), \lambda_2(x), \dots, \lambda_r(x)$ such that:
\begin{equation}\label{Equation x}
x = \sum_{i=1}^r \lambda_i(x) v_i,
\end{equation}
where $r$ is the rank of the Euclidean Jordan algebra $\E$.
The idempotents $ v_i,\,i=1,\dots,r$, satisfy the following properties:
\begin{equation}\label{spectral decomposition}
\sum_{i=1}^r v_i = e, \quad v_i \circ v_j = 0 \quad \text{for}\ \text{all}\  i \neq j, \quad \text{and} \quad v_i\circ v_i = v_i\quad \text{for all}\ i.
\end{equation}
 A set $\{ v_1,v_2,...,v_r \}$ with these properties is called a Jordan frame. For any $x \in \J$, the corresponding eigenvalues $\lambda_{i}(x)$, $i = 1, \ldots, r$, characterize the spectral structure of $x$. An element $x$ belongs to $\K\, (\interior\,(\K))$ if and only if all its eigenvalues are nonnegative (strictly positive) \cite[Proposition 2.5.10]{vieira2007jordan}. Furthermore, both the trace and the determinant of $x$ are determined by its eigenvalues, namely,
\begin{equation}
\tr(x):=\sum_{i=1}^r \lambda_i(x),\quad \det(x):=\displaystyle\prod_{i=1}^r \lambda_i(x).
\end{equation}



The spectral decomposition also provides a straightforward approach to generalizing real-valued continuous functions to  the algebraic framework of symmetric cones:
\begin{itemize}
    \item The square of $x$ is given by:
$
x^{2} := \displaystyle\sum_{i=1}^r \lambda_i(x)^{2} v_i.
$ \footnotemark[1]

    \item The inverse of $x$ is given by:
    $
    x^{-1} := \displaystyle\sum_{i=1}^r \lambda_i(x)^{-1} v_i,
    $ provided that all $\lambda_i(x) \neq 0 $; otherwise, it is undefined.

    \item The square root of $x$ is given by:
    $
    x^{1/2} :=  \displaystyle\sum_{i=1}^r  \lambda_i(x)^{1/2} v_i,
    $ provided that all $\lambda_i(x) \geq 0 $; otherwise, it is undefined.
\end{itemize}
\footnotetext[1]{{With the spectral decomposition~\eqref{Equation x} and the relations~\eqref{spectral decomposition}, we obtain $x\circ x=\sum_{i=1}^{r}\lambda_{i}(x)^{2}v_{i}$.  Hence the present expression for $x^{2}$ is identical to the earlier definition $x^{2}=x\circ x$.
}}

Noting that the Jordan product in a Euclidean Jordan algebra is bilinear, for any element 
$x$, we can associate a Lyapunov operator $\L(x): \J \to \J$ defined by $\L(x)y := x \circ y,\ \forall\, y \in \J.$ The quadratic representation $\P(x): \J \to \J$ is then defined by $\P(x):= 2\L(x)^2 - \L(x^2)$. 

In the following, we present several important properties of the operators $\L(x)$ and $\P(x)$, which are also available in \cite[Proposition 2.3.4, Proposition 2.9.7, Theorem 2.9.8, Proposition 9.2.11]{vieira2007jordan}.

\begin{proposition}\label{Prop:property-of-p-and-l}
If $x\in \J$ is invertible, then the following identities hold: 
\[\P(x^{-1})=(\P(x))^{-1}\quad \text{and} \quad \P(x)\L(x^{-1})=\L(x).\]
\end{proposition}

\begin{proposition}
Given a Jordan frame $\{v_1,v_2,...,v_r\}$, define the operators 
$$
    \begin{array}{lll}
      & \P_{ii} = \P(v_i), & i=1,...,r,\\
      & \P_{ij} = 4\L(v_i)\L(v_j), &  i,j=1,...,r,\ i\neq j.
    \end{array}
$$
With these operators, the space $\J$ admits the orthogonal direct-sum decomposition
$$
    \J = \bigoplus\limits_{1\le i\le j \le r} \V_{ij},
$$
where each $\P_{ij}$ is the orthogonal projection onto the subspace $\V_{ij}$, and they are orthogonal to each other.
\end{proposition}


\begin{proposition}\label{Prop:eig value of Lyapunov}
Let $x \in \J$ have Jordan eigenvalues $\lambda_1(x), \dots, \lambda_r(x)$ with respect to a Jordan frame $\{v_i\}_{i=1}^r$. Then the Lyapunov operator $\L(x)$ and quadratic representation $\P(x)$ admit the following spectral decompositions: 
$$
\begin{array}{ll}
     &      \L (x) = \sum\limits_{i=1}^r \lambda_i(x) \P_{ii} + \sum\limits_{i<j} \frac{\lambda_i(x) + \lambda_j(x)}{2} \P_{ij},\\
     &     \P (x) = \sum\limits_{i=1}^r \lambda_i(x)^2 \P_{ii} + \sum\limits_{i<j} \lambda_i(x) \lambda_j(x)\P_{ij}.
\end{array}
$$
\end{proposition}

Proposition~\ref{Prop:eig value of Lyapunov} immediately yields the following corollary.
\begin{corollary}\label{Cor:invertible}
    If $x\in \interior\,(\K)$, then both $\L(x)$ and $\P(x)$ are invertible.
\end{corollary}

\subsection{Self-concordance}
Self-concordant functions are an important class of functions in convex optimization and play a central role in deriving iteration complexity bounds in certain optimization algorithms. 



\begin{definition}
	Let $\Omega$ be a non-empty open convex domain of $\J$.
	A convex function $\phi(x)$ is said to be $\alpha$-self-concordant if $\phi(x)$ is three times continuously differentiable on $\Omega$, and
	\begin{equation}
		\left|D^3\phi(x)[h,h,h]\right|\leq2\alpha^{-\frac{1}{2}}\left(D^2\phi(x)[h,h]\right)^{\frac{3}{2}}
	\end{equation}
    for all $x\in \Omega$ and all $h\in \J$. If $\alpha=1$, the function $\phi$ is called standard self-concordant.
\end{definition}
Building upon the concept of self-concordance, Nesterov and Nemirovski \cite{nesterov1994interior} introduced the family of strongly self-concordant functions, which exhibit additional properties.
\begin{definition}\label{def:self-concordant family}
	The family $\{\phi(x,\mu)\mid \mu>0\}$ is called strongly self-concordant on a non-empty open convex domain $\Omega$ of $\J$ with the parameter functions $\alpha(\mu)$, $\beta(\mu)$, $\gamma(\mu)$, $\xi(\mu)$, and $\zeta(\mu)$, where $\alpha(\mu),\,\beta(\mu),$ and $\gamma(\mu)$ are continuously
	differentiable, if the following properties hold:
	\begin{enumerate}
		\item[(i)] Convexity and differentiability: $\phi(x,\mu)$ is convex in $x$, continuous in $(x,\mu)\in \Omega\times\mathbb{R}_{++}$, three times continuously differentiable in $x$, and twice continuously differentiable in $\mu$.
		\item[(ii)] Self-concordance of members: $\forall \mu>0$, $\phi(x,\mu)$ is $\alpha(\mu)$-self-concordant
		on $\Omega$.
		\item[(iii)] Compatibility of neighbours: For every $(x,\mu) \in \Omega\times\mathbb{R}_{++}$, $h \in \J$, $\nabla_x\phi(x,\mu)$ and $\nabla_x^2\phi(x,\mu)$  are continuously differentiable in $\mu$, and
		\begin{subequations}
			\begin{align*}
				&| \langle h, \nabla_x\phi^\prime (x,\mu) \rangle- \{{\ln}\beta(\mu)\}_{\mu}^{\prime}\langle h,\nabla_{x}\phi(x,\mu)\rangle| 	\leq\xi(\mu)(\alpha(\mu) D_{x}^{2}\phi(x,\mu)[h,h])^{\frac{1}{2}}, \\
				&| D_x^2\phi'(x,\mu)[h,h]- \{{\ln}\gamma(\mu)\}_{\mu}^{\prime}  D_{x}^{2}\phi(x,\mu)[h,h]|\leq2\zeta(\mu)D_{x}^{2}\phi(x,\mu)[h,h]. 
			\end{align*}
		\end{subequations}
	\end{enumerate}
\end{definition}

 In addition, $\phi$ is a barrier of $\Omega$ if $\phi(x)\rightarrow\infty$ along every sequence of points $x\in \Omega$ converging to a boundary point of $\Omega$. When $\Omega=\interior\,(\K)$ for a closed convex cone $\K$, a barrier $\phi(x)$ is called a $\nu$-logarithmically homogeneous self-concordant barrier ($\nu$-$\LHSC$) if it is standard self-concordant and satisfies
    \begin{equation}\label{log homogeneous}
	\phi(\tau x)= \phi(x)-\nu\ln(\tau),\quad \forall\, x\in \interior\,(\K),\;\forall\, \tau>0
    \end{equation} 
    for some parameter $\nu>0$.
Some straightforward consequences of \eqref{log homogeneous} include
\begin{subequations}
	\begin{align}
		\nabla\phi(\tau x)& =\frac{1}{\tau}\nabla\phi(x),\quad \nabla^2\phi(\tau x)=\frac{1}{\tau^2}\nabla^2\phi(x), \\
		\nabla^2\phi(x)x& =-\nabla\phi(x),\quad 
		\langle \nabla\phi(x),x\rangle  =-\nu. \label{equation 1}
	\end{align}
\end{subequations}

A $\nu$-$\LHSC$ of $\K$ plays a crucial role not only in the convergence analysis of IPMs but also in deriving significant geometric results. The following lemma, as stated in \cite[Proposition 3.3]{chua2013barrier}, proves useful in explaining the proposed method.
\begin{lemma}\label{limsup K}
	For any fixed $\bar{x}\in \K$, let $\phi$ be a $\nu$-$\LHSC$ of $\K$. Then 
$$
    N_{\K} (\bar{x}) = \limsup\limits_{\substack{\mu \downarrow 0 \\ x\in {\rm int}\, (\K) \rightarrow \bar{x}}} \left\{  \nabla \phi (\frac{x}{\mu})\right\}, 
$$
where $N_{\K}(\bar{x})$ denotes the normal cone to $\K$ at $\bar{x}$.
\end{lemma}

For any symmetric cone $\K$ in a Euclidean Jordan algebra $\J$, there exists a natural barrier function \cite[Section 2.6]{vieira2007jordan} of $\K$, defined by 
\begin{equation}\label{def:natural barrier}
    \phi(x): = -\ln (\det(x)), \quad x\in \interior\,(\K).
\end{equation}
It follows immediately that the natural barrier $\phi(x)$ is infinitely differentiable and strictly convex on $\interior\,(\K)$, and that it is a $\nu$-$\LHSC$, where $\nu$ denotes the rank of the symmetric cone $\K$. In particular, its first- and second-order derivatives admit the explicit forms (see \cite[Section 2.6]{vieira2007jordan} for details)
\begin{equation}\label{eq:derivatives of phi}
    \nabla \phi(x)=-x^{-1},\quad \nabla^2 \phi(x)=\P(x)^{-1}.
\end{equation}
A direct calculation shows that, for the natural barrier $\phi$ on a symmetric cone $\K$, its  Fenchel conjugate defined by 
$\phi^*(s):= \sup \{ -\langle s,x \rangle - \phi(x) \mid x\in {\rm int}\, (\K) \}$ 
satisfies $\phi^*(s)=\phi(s)-\nu$ for all $s\in \interior\,(\K^*)=\interior\,(\K)$. 

 For completeness, we recall the notion of monotone operators, which will be used in the convergence analysis. 
\begin{definition}
    Let $\mathbb{H}$ be a Hilbert space with inner product $\langle \cdot,\cdot \rangle$.
	A set-valued mapping $\mathcal{T}: \mathbb{H}\rightrightarrows \mathbb{H}$ is said to be a monotone operator if 
	$$
		\forall\, (x,y), \ (\tilde{x}, \tilde{y}) \in {\rm gph}\, \mathcal{T}, \quad \langle x - \tilde{x}, y - \tilde{y} \rangle \ge 0,
	$$
	where ${\rm gph}\, \mathcal{T} : = \left\{ (x,y)\in \mathbb{H}^2 \mid  y\in \mathcal{T}(x) \right\}$ denotes the graph of $\mathcal{T}$. A monotone operator is maximal monotone if its graph is not properly contained in the graph of another monotone operator, i.e., for all
	$(x,y)\in \mathbb{H}^2$, if 
	$$
		\inf\limits_{(\tilde{x},\tilde{y})\in {\rm gph}\, \mathcal{T}} \langle x-\tilde{x}, y- \tilde{y} \rangle\ge 0,
	$$
	then $(x,y)\in {\rm gph}\, \mathcal{T}$.
\end{definition}



\section{A novel augmented Lagrangian function}\label{sec2}
In this section, we introduce a novel augmented Lagrangian function and present its relevant properties, which play an important role in the design and complexity analysis of the proposed method. 



For any barrier parameter $\mu>0$, if we employ the natural barrier of $\K$ to enforce the cone constraints in \eqref{cone problem}, the corresponding barrier subproblems are defined as:
\begin{equation}\label{cone problem with barrier}
\qquad\begin{alignedat}{2}
\mathllap{\vcenter{\hbox{$(\operatorname{P_\mu})$}}\quad} &
\begin{aligned}[c]
\min\ & \langle c, x\rangle + \mu\,\varnewphi(\frac{x}{\mu}) - \mu \nu\\
\text{s.t.}\ & \mathcal{A}x=b,
\end{aligned}
\qquad\qquad\qquad
\mathllap{\vcenter{\hbox{$(\operatorname{D_\mu})$}}\quad} &
\begin{aligned}[c]
\min\ & -\langle b, \lambda \rangle + \mu\,\phi(s)\\
\text{s.t.}\ & \mathcal{A}^{*}\lambda + s = c.
\end{aligned}
\end{alignedat}
\end{equation}
The feasible region of  ($\operatorname{P_{\mu}}$) from (\ref{cone problem with barrier}) is denoted by $\mathcal{F}_{\operatorname{P_{\mu}}}$.

 Applying the proximal point method to  ($\operatorname{P_{\mu}}$) then yields the following subproblem
\begin{equation}\label{augmented Lagrangian trans}
\begin{aligned}
    &\ \min\limits_{\tilde{x}} \ \langle c,\tilde{x}\rangle + \mu  \varnewphi (\frac{\tilde{x}}{\mu}) -\nu \mu + \frac{\rho}{2} \Vert \tilde{x} - x \Vert^2 + \mathcal{I}_{\mathcal{F}_{\operatorname{P_{\mu}}}}(\tilde{x})\\
    = &\ - \max\limits_{\tilde{x}} \left\{ \min\limits_{\lambda, s} l(\tilde{x},\lambda,s;\mu) -\frac{\rho}{2} \Vert \tilde{x} - x \Vert^2\right\}\\
    =  &\ -\min\limits_{\lambda,s} \left\{ \max\limits_{\tilde{x}} \left\{ -\langle b, \lambda \rangle +\mu \phi(s) + \langle \tilde{x}, \mathcal{A}^* \lambda +s-c \rangle -\frac{\rho}{2} \Vert \tilde{x}-x\Vert^2  \right\} \right\}\\
    = &\ - \min\limits_{\lambda,s}  \left\{ -\langle b, \lambda \rangle + \mu \phi(s) +  \langle x, \mathcal{A}^* \lambda +s - c \rangle + \frac{1}{2\rho} \Vert \mathcal{A}^* \lambda +s-c \Vert^2 \right\}.
    \end{aligned}
\end{equation}
Here, $l(\tilde{x},\lambda,s;\mu) : = -\langle b, \lambda \rangle +\mu \phi(s)+ \langle \tilde{x}, \mathcal{A}^* \lambda+s-c\rangle$  is the Lagrangian function of ($\operatorname{D_{\mu}}$), $\mathcal{I}_{\mathcal{F}_{\operatorname{P_{\mu}}}}$ denotes the indicator function of $\mathcal{F}_{\operatorname{P_{\mu}}}$, and $\rho>0$ is the penalty parameter. This derivation naturally leads to an augmented Lagrangian function for the barrier subproblem~\eqref{cone problem with barrier} as follows: 
\begin{equation}\label{augmented Lagrangian}
	L(x,\lambda,s;\mu,\rho):=-\rho \langle{b}, \lambda \rangle+\rho\mu\phi(s)+\rho\langle x,\mathcal{A}^{*}\lambda+s-c\rangle+\dfrac{1}{2}\left\|\mathcal{A}^{*}\lambda+s-c\right\|^2.
\end{equation} 
\begin{remark}
    For any fixed $(x,\lambda,\mu,\rho)$, the function $L(x,\lambda,\cdot;\mu,\rho)$ is $1$-strongly convex in ${\rm int}\, (\K)$, and $L(x,\lambda,s;\mu,\rho) \rightarrow +\infty$ as $s$ converges to a boundary point of $\K$. Therefore, it admits a unique global minimizer in ${\rm int}\, (\K)$.
\end{remark}

Let $s(x,\lambda;\mu,\rho)$ be the unique global minimizer of $L(x,\lambda,\cdot;\mu,\rho)$ in ${\rm int}\, (\K)$. 
Then $\nabla_s L(x,\lambda,s(x,\lambda;\mu,\rho);\mu,\rho) = 0$, i.e.,
\begin{equation}\label{quadratic function}
 \rho\mu\nabla_s\phi(s(x,\lambda;\mu,\rho))+\rho x+\mathcal{A}^{*}\lambda+s(x,\lambda;\mu,\rho)-c=0.
\end{equation}
 This combined with the identity $\nabla_s \phi(s(x,\lambda;\mu,\rho)) = -s(x,\lambda;\mu,\rho)^{-1}$ indicates that $s(x,\lambda;\mu,\rho)$ admits a closed form expression 
\begin{equation}\label{Eq:function s}
		s(x,\lambda;\mu,\rho)=\frac{1}{2}\left(\left((\rho x-c+\mathcal{A}^{*}\lambda)^2+4\rho\mu e\right)^{\frac{1}{2}}-(\rho x-c+\mathcal{A}^{*}\lambda)\right),
\end{equation}
where both the squaring and square root operations are defined within the context of Euclidean Jordan algebra,  and $e$ denotes its identity element.
Since $\phi$ is infinitely differentiable on $\interior\,(\K)$, then
$
    \nabla_s L(x,\lambda,s;\mu,\rho) 
$
is infinitely differentiable with respect to $(x,\lambda,s,\mu,\rho)$ and 
$$
    \nabla_{s}^2 L(x,\lambda,s;\mu,\rho) = \rho \mu \nabla^2_{s} \phi (s) \succ 0\quad \text{if } s\in \interior\, (\K) .
$$
By the implicit function theorem, $s(x,\lambda;\mu,\rho)$ is therefore infinitely differentiable with respect to $(x,\lambda,\mu,\rho)$.

 In the computation of the NAL method (see Algorithm~\ref{algsbal}), each outer iteration requires approximately solving a subproblem that minimizes $L(x,\lambda,s;\mu,\rho)$ with respect to $(\lambda,s)$, whose search space dimension is significantly larger than that of the classical ALM. However, since $s(x,\lambda;\mu,\rho)$ admits a closed form expression, we can instead directly consider the function obtained by minimizing with respect to $s$. We refer to this resulting function as a \textit{novel augmented Lagrangian function}, defined by 
\begin{equation}
	\begin{aligned}
	\eta(x,\lambda;\mu,\rho) &:= \min_s\left\{L(x,\lambda,s;\mu,\rho)\right\}\\
	&\ =  -\rho \langle b,\lambda\rangle+\rho\mu\phi(s(x,\lambda;\mu,\rho))+\rho \langle x,\mathcal{A}^{*}\lambda+s(x,\lambda;\mu,\rho)-c\rangle\\
	& \qquad \qquad +\dfrac{1}{2}\left\|\mathcal{A}^{*}\lambda+s(x,\lambda;\mu,\rho)-c\right\|^2.
	\end{aligned}
\end{equation}
 The function $\eta(x,\lambda;\mu,\rho)$ is then infinitely differentiable with respect to $(x,\lambda,\mu,\rho)$, derived from the infinite differentiability of $s$.
Theorem~\ref{self concordant} demonstrates that $\eta(x,\lambda;\mu,\rho)$ is self-concordant with respect to $\lambda$.

\begin{theorem}\label{self concordant}
    Fix $\mu>0$, $\rho>0$, and $x\in\J$. Assume that $\rho\mu<1$. Then the function $ \eta(x,\cdot;\mu,\rho)$ is $\rho\mu$-self-concordant.
\end{theorem}
\begin{proof}
	 We first show that $L(x,\cdot,\cdot;\mu,\rho)$ is $\rho \mu$-self-concordant. Let $t\!=\!(\lambda,s)$, and let $ h=(h_1,h_2) \in \mathbb{R}^m\times \J$  be an arbitrary vector.
     Define $$L_1(x,\lambda,s;\mu,\rho)=-\rho  \langle b,\lambda\rangle+\rho\mu\phi(s)+\rho \langle x,\mathcal{A}^{*}\lambda+s-c\rangle$$ and $$L_2(\lambda,s;\mu,\rho)=\dfrac{1}{2}\left\|\mathcal{A}^{*}\lambda+s-c\right\|^2.$$ 
     Then $L(x,\lambda,s;\mu,\rho)=L_1(x,\lambda,s;\mu,\rho)+L_2(\lambda,s;\mu,\rho)$. Direct computation shows that 
     $$
     D^2_tL_1 (x,\lambda,s;\mu,\rho)[h,h] = \rho\mu D^2_{s} \phi(s)[h_2,h_2] \ge 0,
     $$ 
     which implies $L_1 (x,\cdot,\cdot,\mu,\rho)$ is convex. Moreover,
     \begin{equation}
	 	\begin{aligned}
	 		\left|D_{t}^3L_1(x,\lambda,s;\mu,\rho)[h,h,h]\right|=&\rho\mu\left|D_{s}^3\phi(s)[h_2,h_2,h_2]\right|
	 		\leq2\rho\mu({D_{s}^2\phi(s)[h_2,h_2]})^{\frac{3}{2}}\\
	 		=&\dfrac{2}{\sqrt{\rho\mu}}({D_{t}^2L_1(x,\lambda,s;\mu,\rho)[h,h]})^{\frac{3}{2}}.
	 	\end{aligned}
	 \end{equation} 
     Thus, $L_1(x,\cdot,\cdot;\mu,\rho)$ is $\rho\mu$-self-concordant. Furthermore, 
     $$
     D_{t}^2L_2(\lambda,s;\mu,\rho)[h,h] = \| \A^* h_1 + h_2 \|^2 \ge 0,
     $$ 
     which shows that $L_2(\cdot,\cdot;\mu,\rho)$ is convex. In addition,
      $$
      \left|D_{t}^3L_2(\lambda,s;\mu,\rho)[h,h,h]\right|=0\leq2({D_{t}^2L_2(\lambda,s;\mu,\rho)[h,h]})^{\frac{3}{2}}.
      $$
      By \cite[Theorem 5.1.1]{nesterov2018lectures}, it follows that
      $$
      \left|D_{t}^3L(x,\lambda,s;\mu,\rho)[h,h,h]\right|\leq\dfrac{2}{\sqrt{\rho\mu}}({D_{t}^2L(x,\lambda,s;\mu,\rho)[h,h]})^{\frac{3}{2}}.
      $$
      Hence, $L(x,\cdot,\cdot;\mu,\rho)$ is $\rho\mu$-self-concordant.

	 Since 
	\begin{equation}
		\begin{aligned}
			\eta(x,\lambda;\mu,\rho)&=\min_s\left\{L(x,\lambda,s;\mu,\rho)\right\},
		\end{aligned}
	\end{equation} 
     \cite[Theorem 5.1.11]{nesterov2018lectures} further implies that $$\left|D_{\lambda}^3\eta(x,\lambda;\mu,\rho)[h_1,h_1,h_1]\right|\leq{\dfrac{2}{\sqrt{\rho\mu}}}\left(D_{\lambda}^2\eta(x,\lambda;\mu,\rho)[h_1,h_1]\right)^{\frac{3}{2}}.$$
	The proof is complete.
\end{proof}


 To derive the explicit expressions for the first- and second-order derivatives of $\eta(x,\lambda;\mu,\rho)$, we introduce an auxiliary variable
\begin{equation}\label{equation z}
	z(x,\lambda;\mu,\rho) := \rho x+\mathcal{A}^{*}\lambda+s(x,\lambda;\mu,\rho)-c.
\end{equation}
 Since $s(x,\lambda;\mu,\rho)$ is infinitely differentiable, it follows immediately that $z(x,\lambda;\mu,\rho)$ is infinitely differentiable with respect to $(x,\lambda,\mu,\rho)$. Both variables also satisfy the following relationship.
\begin{lemma}\label{center path}
Given  $\mu>0$, $\rho >0$, $x\in \J$, and $(x,\lambda)\in \J \times \mathbb{R}^m$, then the following identities hold:
	\begin{subequations}
	\begin{align}
	  &\langle s(x,\lambda;\mu,\rho),z(x,\lambda;\mu,\rho)\rangle=\rho\mu\nu,\\
	  &  s(x,\lambda;\mu,\rho) \circ z(x,\lambda;\mu,\rho) = \rho \mu e.\label{eq:center path b}
	\end{align}
	\end{subequations}
\end{lemma}
\begin{proof}
Based on the definitions of $s(x,\lambda;\mu,\rho)$ and $z(x,\lambda;\mu,\rho)$, we have
	$$
	\begin{aligned}
	  	\langle & s(x,\lambda;\mu,\rho),z(x,\lambda;\mu,\rho)\rangle=\langle s(x,\lambda;\mu,\rho),-\rho\mu\nabla_s \phi(s(x,\lambda;\mu,\rho))\rangle \overset{\text{(i)}}{=} \rho\mu\nu,\\
	  	& s(x,\lambda;\mu,\rho) \circ z(x,\lambda;\mu,\rho) = s(x,\lambda;\mu,\rho)\circ (-\rho \mu \nabla_s \phi(s(x,\lambda;\mu,\rho)) ) \overset{\text{(ii)}}{=} \rho \mu e. 
	\end{aligned}
    $$ 
Here, identity~(i) follows from \eqref{equation 1}, and identity~(ii) follows from \eqref{eq:derivatives of phi}.
\end{proof}
\begin{remark}
It is noteworthy that in classical IPMs, the iterates $x$ and $s$ generally only satisfy $\langle x,s\rangle\approx \mu\nu$ \cite{nocedal1999numerical}, thereby departing from the central path. In contrast, Lemma~\ref{center path} demonstrates that $s(x,\lambda;\mu,\rho)$ and $z(x,\lambda;\mu,\rho)$ are exactly on the central path during iterations, which distinguishes the proposed method from classical IPMs.
\end{remark}
In addition, by defining $W := \rho\mu \nabla_{s}^{2} \phi(s(x,\lambda;\mu,\rho))$ and $H:=I+W$, we obtain the derivatives of $s(x,\lambda;\mu,\rho)$ and $z(x,\lambda;\mu,\rho)$ in Lemma~\ref{derivatives}, which are fundamental for subsequent computations.
\begin{lemma}\label{derivatives}
 	For given $\mu>0$ and $\rho >0$, the following conclusions hold:
 	\begin{enumerate}
 		\item[(i)] The derivatives of $s(x,\lambda;\mu,\rho)$ and $z(x,\lambda;\mu,\rho)$ with respect to $x$ are
 		\begin{subequations}
 			\begin{align}
 				D_{x} s(x,{\lambda};\mu,\rho)&=-\rho H^{-1},\label{derivatives s of x}\\
 				D_{x} z(x,{\lambda};\mu,\rho)&=\rho\left(I-H^{-1}\right).\label{derivatives z of x}
 			\end{align}
 		\end{subequations}
 		\item[(ii)]The derivatives of $s(x,\lambda;\mu,\rho)$ and $z(x,\lambda;\mu,\rho)$ with respect to $\lambda$ are
 		\begin{subequations}
 			\begin{align}
 				D_{\lambda} s(x,{\lambda};\mu,\rho)&=-H^{-1}\mathcal{A}^*,\label{derivatives s of lambda}\\
 				D_{\lambda} z(x,{\lambda};\mu,\rho)&=(I-H^{-1})\mathcal{A}^*.\label{derivatives z of lambda}
 			\end{align}
 		\end{subequations}  
 	  \item[(iii)] The derivatives of $s(x,\lambda;\mu,\rho)$ and $z(x,\lambda;\mu,\rho)$ with respect to $\mu$ are
 	   \begin{subequations}
 	  	\begin{align}
 	  		s'(x,{\lambda};\mu,\rho)&=-\rho H^{-1}\nabla_s \phi(s(x,\lambda;\mu,\rho)),\label{derivatives s of mu}\\
 	  		z'(x,{\lambda};\mu,\rho)&=-\rho H^{-1}\nabla_s \phi(s(x,\lambda;\mu,\rho)).\label{derivatives z of mu}
 	  	\end{align}
 	  \end{subequations}
 	\end{enumerate}
 \end{lemma}

\begin{proof}
Note that 
\begin{equation}\label{equation s1}
	z(x,\lambda;\mu,\rho)=\rho x+\mathcal{A}^{*}\lambda+s(x,\lambda;\mu,\rho)-c
\end{equation}
and 
\begin{equation}\label{equation s2}
	z(x,\lambda;\mu,\rho)=-\rho\mu\nabla_s \phi(s(x,\lambda;\mu,\rho)).
\end{equation}
 By differentiating both sides of \eqref{equation s1}
 and \eqref{equation s2}
 with respect to $x$, we obtain 
 \begin{equation}
	 \rho I+D_x s(x,\lambda;\mu,\rho) =-\rho\mu \nabla_{s}^{2} \phi(s(x,\lambda;\mu,\rho))D_x s(x,\lambda;\mu,\rho).
	 \end{equation}
 Therefore, \eqref{derivatives s of x} and \eqref{derivatives z of x}  hold. Equations \eqref{derivatives s of lambda}, \eqref{derivatives z of lambda}, \eqref{derivatives s of mu} and \eqref{derivatives z of mu} can be obtained in the same way.
\end{proof}

Lemma~\ref{derivatives} immediately yields the following result.
\begin{theorem}\label{gradent and hessian of eta}
	The gradient and Hessian of $\eta(x,\lambda;\mu,\rho)$ with respect to $\lambda$ are
		\begin{subequations}\label{implements of Ly2}
		\begin{align}
			\nabla_\lambda \eta(x,\lambda;\mu,\rho) &= -\rho b+ \mathcal{A}z(x,\lambda;\mu,\rho),\label{gradient of Ly2}\\
			\nabla_{\lambda }^2 \eta(x,\lambda;\mu,\rho) &=\mathcal{A}H^{-1}W\mathcal{A}^{*}.\label{hessian of Ly2}
		\end{align}
	\end{subequations}
The gradient and Hessian of $\eta(x,\lambda;\mu,\rho)$ with respect to $x$ are
    	\begin{subequations}\label{gradient of eta with x}
    	\begin{align}
    		\nabla_x \eta(x,\lambda;\mu,\rho) &=\rho(\mathcal{A}^*\lambda+s(x,\lambda;\mu,\rho)-c) ,\label{gradient of Lx2}\\
    		\nabla_{x}^2 \eta(x,\lambda;\mu,\rho) &=-\rho^2H^{-1}.\label{hessian of Lx2}
    	\end{align}
    \end{subequations}
Therefore, $\eta(x,\lambda;\mu,\rho)$ is strictly convex in $\lambda$ and strictly concave in $x$.
\end{theorem}

More properties of \( \eta(x,\lambda;\mu,\rho) \) are given in Theorem~\ref{inequality 14}, which outlines the sensitivity of \( \nabla_{\lambda} \eta(x,\lambda;\mu,\rho) \) and \( \nabla_{\lambda}^2 \eta(x,\lambda;\mu,\rho) \) with respect to \( \mu \).


\begin{theorem}\label{inequality 14}
	For given $\rho > 0$, $\mu > 0$, and $(x,\lambda) \in \J \times \mathbb{R}^m$, we have
	\begin{subequations}\label{inequality 14-1}
    	\begin{align}
    		& \left|\nabla_{\lambda} \eta'(x,\lambda;\mu,\rho)^{\top}h\right|\leq\sqrt{\dfrac{\rho\nu}{\mu}}\sqrt{h^{\top}\nabla_{\lambda }^2 \eta(x,\lambda;\mu,\rho)h},\,\forall\, h\in \mathbb{R}^m,  \\
    		& \left|h^{\top}\nabla_{\lambda}^2 \eta'(x,\lambda;\mu,\rho)h\right|\leq{\dfrac{1+2\sqrt{\nu}}{\mu}\left|h^{\top}\nabla_{\lambda}^2 \eta (x,\lambda;\mu,\rho)h\right|,\,\forall\, h\in \mathbb{R}^m. }
    	\end{align}
    \end{subequations}
\end{theorem}
\begin{proof}
	According to Theorem~\ref{gradent and hessian of eta} and \eqref{derivatives z of mu}, we have 
	\begin{equation*}
		\nabla_{\lambda} \eta'(x,\lambda;\mu,\rho)=\mathcal{A}z'(x,\lambda;\mu,\rho)=-\rho \mathcal{A}H^{-1}\nabla_s \phi(s(x,\lambda;\mu,\rho)).
	\end{equation*}
	Therefore,
	\begin{equation}
		\begin{aligned}
			&\ \left|\nabla_{\lambda} \eta'(x,\lambda;\mu,\rho)^{\top}h\right|\\
			=&\ \rho\left|h^{\top} \mathcal{A}H^{-1}\nabla_s \phi(s(x,\lambda;\mu,\rho))\right|\\
			=&\ \rho\left|h^{\top} \mathcal{A}H^{-\frac{1}{2}}W^{\frac{1}{2}}W^{-\frac{1}{2}}H^{-\frac{1}{2}}\nabla_s \phi(s(x,\lambda;\mu,\rho))\right|\\
			\leq &\ \rho\left\|h^{\top}  \mathcal{A}H^{-\frac{1}{2}}W^{\frac{1}{2}}\right\|_2\left\|W^{-\frac{1}{2}}H^{-\frac{1}{2}}\nabla_s \phi(s(x,\lambda;\mu,\rho))\right\|_2\\
				\overset{\text{(i)}}{\leq} &\ \rho\sqrt{h^{\top}\nabla_{\lambda }^2 \eta(x,\lambda;\mu,\rho)h}\sqrt{\langle \nabla_s \phi(s(x,\lambda;\mu,\rho)), W^{-1}\nabla_s \phi(s(x,\lambda;\mu,\rho))\rangle}\\
			=&\ \sqrt{\dfrac{\rho\nu}{\mu}}\sqrt{h^{\top}\nabla_{\lambda }^2 \eta(x,\lambda;\mu,\rho)h},
		\end{aligned}
	\end{equation}
	where (i) follows from the facts that  $H^{-\frac{1}{2}} W H^{-\frac{1}{2}} = H^{-1}W$ and $ H^{-\frac{1}{2}}W^{-1} H^{-\frac{1}{2}} \prec W^{-1}$. This proves (\ref{inequality 14-1}a). Define
	\begin{equation}
		\psi(\mu) := h^{\top}\nabla_{\lambda}^2 \eta (x,\lambda;\mu,\rho)h.
	\end{equation}
	For simplicity, we will abbreviate $\phi(s(x,\lambda;\mu,\rho))$ and $s(x,\lambda;\mu,\rho)$ as $\phi$ and $s$,  respectively. 
	From Theorem~\ref{gradent and hessian of eta}, we have
	\begin{equation}\label{function psi}
		\psi(\mu)=h^{\top}\mathcal{A}H^{-1}W\mathcal{A}^{*}h={h}^{\top}\mathcal{A}(I+W^{-1})^{-1}\mathcal{A}^{*}{h}.
	\end{equation}
	Differentiating \eqref{function psi} with respect to $\mu$ yields:
	\begin{equation}\label{function dphi}
		\begin{aligned}
			\psi'(\mu)&=h^{\top}\mathcal{A}(I+W^{-1})^{-1}W^{-1}(\rho\nabla_{s}^2\phi+\rho\mu D_{s}^3\phi s')W^{-1}(I+W^{-1})^{-1}\mathcal{A}^{*}{h}\\
			&=\frac{1}{\mu} \langle \bar{h},W\bar{h}\rangle+(\rho\mu D_{s}^3\phi)[\bar{h},\bar{h}, s'],
		\end{aligned}
	\end{equation}
	where $\bar{h}:=H^{-1}\mathcal{A}^{*}{h}$.
	On the one hand,
	\begin{equation}\label{function dphi1}
		\begin{aligned}
			\left|\frac{1}{\mu}\langle \bar{h},W\bar{h}\rangle \right|&=	\frac{1}{\mu}\left|{h}^{\top}\mathcal{A}H^{-1}WH^{-1}\mathcal{A}^{*}{h}\right|\\
			& \overset{\text{(i)}}{\leq} \frac{1}{\mu}\left|{h}^{\top}\mathcal{A}H^{-1}W\mathcal{A}^{*}{h}\right|\\
			&=\frac{1}{\mu}\left|h^{\top}\nabla_{\lambda}^2 \eta (x,\lambda;\mu,\rho)h\right|,
		\end{aligned}
	\end{equation}
	 where (i) follows from the fact that $H^{-1} W H^{-1} \prec H^{-1} W$. On the other hand, applying \cite[Lemma 5.1.2]{nesterov2018lectures} to the last term of \eqref{function dphi}, we have
	\begin{equation}\label{function dphi2}
		\begin{aligned}
			\left|(\rho\mu D_{s}^3\phi)[\bar{h},\bar{h}, s']\right|&\leq 2\rho\mu D_{s}^2\phi[\bar{h},\bar{h}]\left(D_{s}^2\phi[s',s']\right)^{\frac{1}{2}}\\
			&\leq 2\left|h^{\top}\nabla_{\lambda}^2 \eta (x,\lambda;\mu,\rho)h\right|\sqrt{\rho^2 \langle \nabla_s\phi, H^{-1}\nabla_{s}^2\phi H^{-1}\nabla_s\phi\rangle}\\
			&\leq \dfrac{2}{\mu}\left|h^{\top}\nabla_{\lambda}^2 \eta (x,\lambda;\mu,\rho)h\right|\sqrt{\langle \nabla_s\phi ,(\nabla_{s}^2\phi)^{-1}\nabla_s\phi \rangle}\\
			&=\dfrac{2\sqrt{\nu}}{\mu}\left|h^{\top}\nabla_{\lambda}^2 \eta (x,\lambda;\mu,\rho)h\right|,
		\end{aligned}
	\end{equation}
	where the equality follows from \eqref{equation 1}.  Combining \eqref{function dphi1} with \eqref{function dphi2} yields
	\begin{equation}
		\left|h^{\top}\nabla_{\lambda}^2 \eta'(x,\lambda;\mu,\rho)h\right|=\left|\psi'(\mu)\right|\leq\dfrac{1+2\sqrt{\nu}}{\mu}\left|h^{\top}\nabla_{\lambda}^2 \eta (x,\lambda;\mu,\rho)h\right|,
	\end{equation}
	thereby establishing  (\ref{inequality 14-1}b).
\end{proof}
Based on Theorems~\ref{self concordant} and \ref{inequality 14} with Definition~\ref{def:self-concordant family}, we obtain the following theorem, which demonstrates that the function family $\{\eta(x,\cdot;\mu,\rho)\mid \mu>0\}$ is strongly self-concordant.
\begin{theorem}\label{self concordant family}
	Given that $\mu>0$, $\rho>0$, $\rho\mu<1$, and $x\in \J$, the function family $\{\eta(x,\cdot;\mu,\rho) \mid \mu>0\}$ is strongly self-concordant with parameters
	\begin{equation}\label{thm:stongly-self-concordant-parameter}
		\alpha(\mu)=\rho\mu,\,\beta(\mu)=\gamma(\mu)=1,\,\xi(\mu)=\dfrac{\sqrt{\nu}}{\mu},\,\zeta(\mu)=\dfrac{1+2\sqrt{\nu}}{2\mu}.
	\end{equation}
\end{theorem}
\begin{proof}
     From Theorem~\ref{self concordant}, it suffices to verify property (iii) of the strong self-concordance. By the parameters provided in \eqref{thm:stongly-self-concordant-parameter}, we have 
\begin{equation}
\begin{aligned}
        &\ | \langle h, \nabla_\lambda \eta'(x,\lambda;\mu,\rho) \rangle - \{{\ln}\beta(\mu)\}_{\mu}^{\prime}\langle h,\nabla_{\lambda}\eta (x,\lambda;\mu,\rho)|\\ 
         = &\ 
         | \langle h, \nabla_\lambda \eta'(x,\lambda;\mu,\rho) \rangle |\\
         \overset{\text{(i)}}{\le } & 
        \sqrt{\frac{\rho \nu}{\mu}} \sqrt{h^{\top}\nabla_{\lambda}^2 \eta(x,\lambda;\mu,\rho)h}\\
        = &\ \xi(\mu) \left( \alpha(\mu) D^2_\lambda \eta(x,\lambda;\mu,\rho) [h,h] \right)^{\frac{1}{2}},
\end{aligned}
\end{equation}
and 
\begin{equation}
\begin{aligned}
        &\ | D_\lambda^2\eta'(x,\lambda;\mu,\rho)[h,h]- \{{\ln}\gamma(\mu)\}_{\mu}^{\prime}  D_\lambda^2\eta(x,\lambda;\mu,\rho)[h,h]|\\
         = &\  | D_\lambda^2\eta'(x,\lambda;\mu,\rho)[h,h] |\\
          \overset{\text{(ii)}}{\le} &\
        {\dfrac{1+2\sqrt{\nu}}{\mu}\left|h^{\top}\nabla_{\lambda}^2 \eta (x,\lambda;\mu,\rho)h\right|}\\
       = &\  2\zeta(\mu)D_{\lambda}^{2}\eta(x,\lambda;\mu,\rho)[h,h],
\end{aligned}
\end{equation}
where both (i) and (ii) follow directly from Theorem~\ref{inequality 14}. This completes the proof.
\end{proof}



\section{A Newton augmented Lagrangian method}\label{sec3}
In Section~\ref{sec2}, we presented several favorable properties of the novel augmented Lagrangian function $\eta(x,\lambda;\mu,\rho)$. Building on these results, we now propose a \nal (\NALMM) method based on Problem $(\operatorname{D})$ from \eqref{cone problem}. The convergence analysis and additional computational details are also presented in this section. 




The \NAL for SCP can be summarized in the following framework:

\begin{algorithm}[H]
    \caption{Framework of the NAL method for SCP}
	\begin{algorithmic}
		\State \textbf{Step 1: } Choose $x^{(0)}\in \J$, $\lambda^{(0)}\in\mathbb{R}^m$, $\mu^{(0)}>0$, $\rho^{(0)}>0$, $\sigma\in \left(0,1\right)$. Let $k:=0$.
		\State \textbf{Step 2: } Compute 
		\begin{equation}\label{eq subproblem}
			\lambda^{(k+1)} \thickapprox \argmin\limits_\lambda \eta(x^{(k)},\lambda;\mu^{(k)},\rho^{(k)}) 
		\end{equation}
		\State {\qquad\qquad \, }via Algorithm~\ref{alg4}.
		\State \textbf{Step 3: }  If $x^{(k)}$ is an approximate solution to (\ref{cone problem}), stop the algorithm. \State {\qquad\qquad \,  }Otherwise, go to Step 4.
		\State \textbf{Step 4: } Update $x^{(k+1)}=\frac{1}{\rho^{(k)}}z(x^{(k)},\lambda^{(k+1)};\mu^{(k)},\rho^{(k)})$, $\mu^{(k+1)}=\sigma\mu^{(k)}$, and   
        \State {\qquad\qquad \,  }  $\rho^{(k+1)}=\max\left\{ \frac{1}{2} \rho^{(k)},\rho_{\min} \right\}$. 
		\State \textbf{Step 5:  } Set $k:=k+1$ and go to Step 2. 
		
	\end{algorithmic}\label{algsbal}
\end{algorithm}

The stopping criterion for solving subproblem (\ref{eq subproblem}) is the same as in \cite{rockafellar1976augmented}. For a given sequence $\{ \epsilon^{(k)} \}$ 
with $\epsilon^{(k)}\ge 0$ and $\sum_{k=0}^\infty \epsilon^{(k)} < \infty$, we obtain $\lambda^{(k+1)}$ in Step 2 from Algorithm~\ref{algsbal} at the $k$-th iteration when
		\begin{equation}\label{stopping criteria}
			\eta (x^{(k)},\lambda^{(k+1)};\mu^{(k)},\rho^{(k)}) - 
			\inf\limits_{\lambda} \eta (x^{(k)},\lambda;\mu^{(k)},\rho^{(k)}) \le \frac{1}{2} \rho^{(k)} (\epsilon^{(k)})^2.
		\end{equation}
Note that, in the classical ALM, the stopping criterion  (\ref{stopping criteria}) is difficult to implement because $\inf_{\lambda} \eta (x^{(k)},\lambda;\mu^{(k)},\rho^{(k)})$ is generally unavailable. In contrast, the \NAL can satisfy this criterion numerically via a computable merit function (see Remark~\ref{inexact-lambda}), which makes it implementable in practice, unlike the classical ALM.


It can be observed that the \NAL shares many similarities with the classical ALM. The key distinction, however, is that the ALM does not include the smooth barrier term $\mu \phi(s)$. In this case, Rockafellar demonstrates in his well-known work \cite{rockafellar1976augmented} that for convex cone programming problems, the ALM can be viewed as a proximal point algorithm based on a monotone operator (see \cite[Section 4]{rockafellar1976augmented}). In contrast, the \NAL cannot be viewed as a proximal update of the same monotone operator, due to the presence of the smoothing barrier term. Nevertheless, if one treats the barrier parameter $\mu$ as an additional primal variable, then the \NAL admits an equivalent reformulation as a proximal point algorithm for a suitably defined monotone operator.


From (\ref{augmented Lagrangian trans}), the update of $x^{(k)}$ is equivalent to the proximal point of Problem $(\operatorname{P_{\mu^{(k)}}})$ when the subproblem of Step 2 in Algorithm~\ref{algsbal} is exactly computed. We also update $\mu$ simultaneously. To explain the algorithm, let $\mathcal{S}_k (x^{(k)})$ be the exact solution to 
\begin{equation}\label{proximal point to Pmu}
	\begin{array}{crllcl}
		&\min & \langle{c}, {x}\rangle + \mu^{(k)} \varnewphi(\frac{x}{\mu^{(k)}}) - \mu^{(k)}\nu + \frac{\rho^{(k)}}{2} \Vert x - x^{(k)}\Vert^2\\
		 &\quad \text { s.t.} & \mathcal{A} {x}={b}. \\
		
	\end{array}
\end{equation}
For simplicity, let $\hat{\mu}^{(k)} = \sqrt{{\mu}^{(k)}}$. Define the set-valued mapping $\mathcal{T}$ as 
\begin{equation}
    \mathcal{T} (x,\hat{\mu}): = 
    \left\{\begin{array}{cl}
        \left( c + \nabla \varnewphi ( \frac{x}{\hat{\mu}^2})+N_{\{x \mid \mathcal{A}x=b\}}(x), \hat{\mu}\right), & {\rm if}\ (x,\hat{\mu})\in {\rm int}\, (\K) \times \mathbb{R}_{++};\\
        \left\{c + N_{\K}(x)+N_{\{x \mid \mathcal{A}x=b\}}(x)\right\}\times \mathbb{R}_{-}, & {\rm if}\ (x,\hat{\mu})\in \K \times \{ 0\};\\
        \emptyset, & {\rm other}\ {\rm cases},
    \end{array}\right.
\end{equation}
where $N_{\{x \mid \mathcal{A}x=b\}}(x)$ denotes the normal cone to $\{x \mid \mathcal{A}x=b\}$ at $x$. Then the sequence generated by Algorithm~\ref{algsbal} is given by $\left\{ \left( \mathcal{S}_k(x^{(k)}),\mu^{(k)} \right) \right\}$, such that 
\begin{subequations}\label{augemnted lagrangian method}
	\begin{align}
		&\begin{pmatrix}
			0\\
			0
		\end{pmatrix}
		\in 
		\mathcal{T}(\mathcal{S}_k(x^{(k)}), \hat{\mu}^{(k)})+
		\begin{pmatrix}
			\rho^{(k)}\\
			\frac{\sqrt{\sigma}}{1-\sqrt{\sigma}}
		\end{pmatrix}
		\begin{pmatrix}
			\mathcal{S}_k(x^{(k)})-x^{(k)}\\
			 \hat{\mu}^{(k)}-\hat{\mu}^{(k-1)}
		\end{pmatrix},\\
		& \lambda^{(k+1)} \thickapprox \argmin_{\lambda}\eta(x^{(k)},\lambda;\mu^{(k)},\rho^{(k)})  \label{alm subproblem},\\
		&x^{(k+1)} = x^{(k)} + \dfrac{1}{\rho^{(k)}}(\mathcal{A}^{*} \lambda{^{(k+1)}}+s^{(k+1)}-c),\\
		&\mu^{(k+1)} = (\hat{\mu}^{(k+1)})^2 = \d \sigma (\hat{\mu}^{(k)})^2 = \sigma \mu^{(k)},\\
        & \rho^{(k+1)} = \max \left\{ \frac{1}{2} \rho^{(k)},\rho_{\min} \right\}.
        \end{align}
\end{subequations}

For any $\rho>0$ and $0<\sigma<1$, let $\mathcal{C}_{(\rho,\sigma)}$ be a linear operator from $\J \times \mathbb{R}$ to itself, satisfying 
$\mathcal{C}_{(\rho,\sigma)}(x,\hat{\mu}) = \left( \frac{1}{\rho} x, \frac{1-\sqrt{\sigma}}{\sqrt{\sigma}} \hat{\mu} \right)$.
We now show that under mild conditions, $\mathcal{C}_{(\rho,\sigma)}\mathcal{T}$ is a maximal monotone operator in a certain Hilbert space.

\begin{lemma}\label{maximal monotone operator omega}
	Suppose that $\phi$ is a natural barrier of $\K$. If there exists a constant $\omega>0$ such that 
    \begin{equation}\label{monotone}
      \omega > \frac{\sqrt{\sigma} \nu}{(1- \sqrt{\sigma})\rho },
    \end{equation}
    then $\mathcal{C}_{(\rho,\sigma)}\mathcal{T}$ is maximal monotone on $\J\times \mathbb{R}$, where $\J\times \mathbb{R}$ is endowed with the inner product $\langle \cdot, \cdot \rangle_\omega: \langle (x,\mu), (\tilde{x}, \tilde{\mu}) \rangle_\omega = \langle x, \tilde{x} \rangle+ \omega \langle \mu, \tilde{\mu} \rangle$.
\end{lemma}
\begin{proof}
	It follows from L\"{o}hne's characterization \cite{lohne2008characterization} that $\mathcal{C}_{(\rho,\sigma)}\mathcal{T}(\cdot,\hat{\mu})$ is a maximal monotone operator on $\J$ for any $\hat{\mu}\ge 0$. Thus, by Minty's characterization \cite{bauschke2017correction},  $\mathcal{C}_{(\rho,\sigma)}\mathcal{T}(\cdot,\hat{\mu})+ I_{\J}$ is surjective onto $\J$, where $I_{\J}$ is the identity function on ${\J}$. Then $\mathcal{C}_{(\rho,\sigma)}\mathcal{T}(\cdot,\cdot)+I_{\J\times \mathbb{R}}$ is surjective onto $\J\times \mathbb{R}$ by direct computation. So it suffices to prove that $\mathcal{C}_{(\rho,\sigma)}\mathcal{T}$ is monotone on $\J\times \mathbb{R}$ by Minty's characterization. 
    
    We next show that $\mathcal{C}_{(\rho,\sigma)}\mathcal{T}$ is monotone on ${\rm int}\, (\K) \times \mathbb{R}_{++}$ with inner product $\langle \cdot, \cdot \rangle_\omega$. Let 
    \begin{equation}
        \mathcal{M}(x,\hat{\mu}) := 
        \left\{\begin{array}{cl}
                \left(\frac{1}{\rho}\nabla \varnewphi (\frac{x}{\hat{\mu}^2}),\frac{1-\sqrt{\sigma}}{\sqrt{\sigma}}\hat{\mu}\right), & {\rm if}\ (x,\hat{\mu})\in {\rm int}\, (\K) \times \mathbb{R}_{++};\\
              N_{\K}(x) \times \mathbb{R}_-, & {\rm if}\ (x,\hat{\mu})\in \K \times \{ 0\};\\
              \emptyset, & {\rm other}\ {\rm cases}.
        \end{array}\right.
    \end{equation}
    By definition and the mean value theorem, $\mathcal{M}(\cdot,\cdot)$ is monotone on ${\rm int}\, (\K) \times \mathbb{R}_{++}$ with inner product $\langle \cdot, \cdot \rangle_\omega$ if, for all $\check{h}:=(h,h_0)\in \J \times \mathbb{R}$ and $(x,\hat{\mu}) \in {\rm int}\, (\K) \times \mathbb{R}_{++}$,
    \begin{equation}\label{monotone schur}
    \begin{aligned}
       &\ \langle \check{h}, \frac{1}{2}\left(D_{(x,\hat{\mu})} \mathcal{M}(x,\hat{\mu}) + (D_{(x,\hat{\mu})} \mathcal{M} (x,\hat{\mu}))^* \right) [\check{h}] \rangle_\omega \\
        =&\ \frac{1}{\rho \hat{\mu}^2} \langle h, \nabla^2 \varnewphi(\frac{x}{\hat{\mu}^2})h \rangle - \frac{2 h_0}{\rho \hat{\mu}^3} \langle h, \nabla^2 \varnewphi(\frac{x}{\hat{\mu}^2})x \rangle
       + \frac{1-\sqrt{\sigma}}{\sqrt{\sigma}} \omega h_0^2 \\
        =&\ \frac{1}{\rho \hat{\mu}^2} \langle h - \frac{h_0}{\hat{\mu}} x, \nabla^2 \varnewphi(\frac{x}{\hat{\mu}^2})( h- \frac{h_0}{\hat{\mu}} x)  \rangle+ h_0^2 \left(\omega \frac{1-\sqrt{\sigma}}{\sqrt{\sigma}} -\frac{1}{\rho \hat{\mu}^4} \langle x, \nabla^2 \varnewphi(\frac{x}{\hat{\mu}^2})x \rangle\right),\\
        \ge &\ 0,
    \end{aligned}
    \end{equation}
     where $(D_{(x,\hat{\mu})} \mathcal{M}(x,\hat{\mu}))^* $ is the adjoint operator of $D_{(x,\hat{\mu})} \mathcal{M}(x,\hat{\mu})$. Since
\begin{equation}
\omega \frac{1-\sqrt{\sigma}}{\sqrt{\sigma}} -\frac{1}{\rho \hat{\mu}^4} \langle x, \nabla^2 \varnewphi(\frac{x}{\hat{\mu}^2})x \rangle = \omega \frac{1-\sqrt{\sigma}}{\sqrt{\sigma}} - \frac{\nu}{\rho},
\end{equation}
then the monotonicity holds if 
\begin{equation}
    \omega > \frac{\sqrt{\sigma} \nu}{(1- \sqrt{\sigma}) \rho }.
\end{equation}

To extend monotonicity on $\J\times \mathbb{R}$, it suffices to verify the case when $(x,\hat{\mu})\in \K \times \{ 0\}$ and $(\tilde{x},\tilde{\mu})\in {\rm int}\, (\K) \times \mathbb{R}_{++}$, for any $(y,\tau)\in  \mathcal{M}(x,\hat{\mu})$, $(\tilde{y},\tilde{\tau})\in \mathcal{M}(\tilde{x},\tilde{\mu})$ with $\tau\le 0$ and $\tilde{\tau}> 0$,
$$
    \langle y-\tilde{y}, x-\tilde{x} \rangle+ \omega (\tau -\tilde{\tau})(\hat{\mu}-\tilde{\mu}) \ge 0. 
$$
By Lemma~\ref{limsup K} and the monotonicity of $\mathcal{M}(\cdot,\cdot)$ on ${\rm int}\, (\K) \times \mathbb{R}_{++}$, there exists a sequence $(x^{(i)},\hat{\mu}^{(i)}) \to (x,0)$ with $x^{(i)} \in \interior\, (\K) $ and $\hat{\mu}^{(i)}>0$ such that 
\begin{equation}
    \langle y^{(i)} - \tilde{y}, x^{(i)} - \tilde{x} \rangle + \omega  \frac{1-\sqrt{\sigma}}{\sqrt{\sigma}} (\hat{\mu}^{(i)}- \tilde{\mu})^2 \ge 0,
\end{equation}
where $y^{(i)} = \nabla \varnewphi (\frac{x^{(i)}}{(\hat{\mu}^{(i)})^2})\rightarrow y$.  Taking the limit $i \rightarrow \infty$ yields 
\begin{subequations}\label{monotone operator}
    \begin{align}
        0 & \le  \langle y-\tilde{y},x-\tilde{x} \rangle + \omega \frac{1-\sqrt{\sigma}}{\sqrt{\sigma}} \tilde{\mu}^2\\
         & \le  \langle y-\tilde{y},x-\tilde{x} \rangle + \omega (\frac{1-\sqrt{\sigma}}{\sqrt{\sigma}} \tilde{\mu}-\tau)\tilde{\mu}\\
        & = \langle y-\tilde{y},x-\tilde{x} \rangle + \omega \langle \tau-\tilde{\tau}, 
        0 - \tilde{\mu} \rangle.
    \end{align}
\end{subequations}
Inequality (\ref{monotone operator}b) holds for $\tau\le 0$, and (\ref{monotone operator}c) is the result of $\tilde{\tau} = \frac{1-\sqrt{\sigma}}{\sqrt{\sigma}}\tilde{\mu}$. The set-valued mapping $\mathcal{M}(\cdot,\cdot)$ is then monotone on $\J\times \mathbb{R}$. From the additivity of monotone operators, we obtain that
$\mathcal{C}_{(\rho,\sigma)}\mathcal{T}$ is a maximal monotone operator on $\J\times \mathbb{R}$. This completes the proof.
\end{proof}

\begin{theorem}\label{convergence-x}
    Let $\phi$ be the natural barrier of $\K$. Suppose that $0<\sigma<1$, $\rho^{(k)}$ is bounded from below, and the subproblem at each iteration is executed with the stopping criterion (\ref{stopping criteria}). Then the following conclusions hold:
		\begin{enumerate}
			\item[(i)] The sequence $\{x^{(k)}\}_{k=1}^\infty$ generated by Algorithm
			\ref{algsbal} converges to $\bar{x}$, where $\bar{x}$ is an optimal solution to Problem $(\operatorname{P})$;
			\item[(ii)] The KKT residual norm
			\begin{equation}\label{KKT-residual}
				\max\left\{ \Vert\mathcal{A}^*\lambda^{(k)}+s^{(k)}-c\Vert,\, \Vert \mathcal{A}x^{(k)}-b \Vert,\, \Vert x^{(k)}\circ s^{(k)} \Vert \right\} \rightarrow 0.
			\end{equation}
		\end{enumerate}
\end{theorem}


\begin{proof}
    By the maximal monotonicity of $\mathcal{C}_{(\rho,\sigma)} \mathcal{T}$ and following Rockafellar's results in \cite{rockafellar1976monotone}, it is straightforward to establish this theorem.
\end{proof}

In the following, we present some computational details on the subproblem. It can be observed that the main computational effort of the \NAL lies in solving subproblem~\eqref{eq  subproblem}. We employ the Newton method to solve this subproblem. Specially, the search direction $\Delta\lambda$ is obtained by solving the following equation: 
\begin{equation}\label{Newton direction}
	\nabla_{\lambda }^2 \eta(x,\lambda;\mu,\rho)\Delta\lambda=-\nabla_{ \lambda} \eta(x,\lambda;\mu,\rho).
\end{equation}


 To determine $\nabla_{\lambda }^2 \eta(x,\lambda;\mu,\rho)$ and $\nabla_{ \lambda} \eta(x,\lambda;\mu,\rho)$, we need the explicit expressions of $s(x,\lambda;\mu,\rho)$ and $z(x,\lambda;\mu,\rho)$, given that $\nabla_{ \lambda} \eta(x,\lambda;\mu,\rho)$ depends on $z(x,\lambda;\mu,\rho)$ from Theorem~\ref{gradent and hessian of eta}. Recall that 
\begin{equation}
		s(x,\lambda;\mu,\rho)=\frac{1}{2}\left(\left((\rho x-c+\mathcal{A}^{*}\lambda)^2+4\rho\mu e\right)^{\frac{1}{2}}-(\rho x-c+\mathcal{A}^{*}\lambda)\right),
\end{equation}
 By definition~\eqref{equation z},
 \begin{equation}\label{Eq:function z}
 		z(x,\lambda;\mu,\rho)=\frac{1}{2}\left(\left((\rho x-c+\mathcal{A}^{*}\lambda)^2+4\rho\mu e\right)^{\frac{1}{2}}+(\rho x-c+\mathcal{A}^{*}\lambda)\right).
 \end{equation}
 It is straightforward to verify that $s:=s(x,\lambda;\mu,\rho)$ and $z:=z(x,\lambda;\mu,\rho)$ lie in $\interior\,({\K})$, and share a common Jordan frame by using the spectral decomposition of $s$ and $z$. With these properties, we now derive the explicit expression for the Hessian $\nabla_{\lambda }^2 \eta (x,\lambda;\mu,\rho)$.
\begin{theorem}\label{Theorem:Hessian of SCMs}
    Given the current iterates $s$ and $z$, the Hessian $\nabla_{\lambda }^2 \eta (x,\lambda;\mu,\rho)$ takes the form $$\nabla_{\lambda }^2 \eta (x,\lambda;\mu,\rho) = \A\L(z)(\L(z+s))^{-1}\A^*.
    $$
\end{theorem}
\begin{proof}
Invoking Proposition~\ref{Prop:property-of-p-and-l} and \eqref{eq:center path b}, we obtain  $\P(s) \L(z) = \rho \mu \L(s)$.  
Next, adding $\rho\mu\L(z)$
 to both sides yields 
 \begin{equation}\label{eq:theorem 4.3(1)}
 (\rho\mu I+\P(s)) \L(z)=\rho\mu\L(z+s).
 \end{equation}
 Define $u:=z+s=((\rho x-c+\A^*\lambda)^2+4\rho\mu e)^{\frac{1}{2}}$. By its spectral decomposition, we have $u \in \interior{(\K)}$. 
 Since both $u\in\interior\,(\K)$ and $z\in\interior\,(\K)$, Corollary~\ref{Cor:invertible} implies that $\L(z)$ and $\L(z+s)$ are invertible. Hence, the identity~\eqref{eq:theorem 4.3(1)} can be rewritten in the form $ \rho \mu I+\P(s)= \rho \mu \L(z + s) (\L(z))^{-1}$.
 From the definitions of $W$ and $H$, together with \eqref{eq:derivatives of phi}, we obtain $$WH^{-1}=\rho \mu \P(s)^{-1} \big( \rho \mu \P(s)^{-1} + I \big)^{-1}= \rho \mu \big( \rho \mu I+ \P(s) \big)^{-1}=\L(z)(\L(z+s))^{-1}.$$
 It follows that 
 $$
 \nabla_{\lambda }^2 \eta(x,\lambda;\mu,\rho) =\A WH^{-1}\A^*= \A\L(z)(\L(z+s))^{-1}\A^*,
 $$
 which completes the proof.
 \end{proof}

\begin{remark}\label{condition number remark}
To illustrate the advantage of the \NAL in terms of the condition numbers of its SCMs compared to those from classical IPMs, we use the nonnegative orthant as an example. By Theorem~\ref{Theorem:Hessian of SCMs}, the Hessian $\nabla_{\lambda }^2 \eta(x,\lambda;\mu,\rho)$ in this case takes the form $\A\Diag(z)(\Diag(z+s))^{-1}\A^*$. The diagonal matrix $\Diag(z)(\Diag(z+s))^{-1}$ produced by the \NAL satisfies that each diagonal element lies strictly in the interval $(0, 1)$. In contrast, the corresponding diagonal elements in the classical IPMs are typically in the range \( (0, \infty) \) \cite{nocedal1999numerical}. Consequently, the condition numbers of the SCMs in the \NAL are generally better than those in classical IPMs, which constitutes a significant advantage.
 Furthermore, as shown in Theorem~\ref{thm:condition number}, the condition number of the SCMs in the \NAL is of order $\mathcal{O}(1/{\mu})$, whereas classical IPMs generally yield SCMs whose condition numbers scale as $\mathcal{O}(1/{\mu^2})$ \cite[Lemma 4.1]{gondzio2012matrix}. This improvement is also demonstrated via a heatmap in our numerical experiments, providing a more intuitive comparison.
\end{remark}



\begin{theorem}\label{thm:condition number}
Suppose  $\{ \rho^{(k)}\}$ is bounded from below. Then the condition number of the  Hessian $\nabla^2_{\lambda} \eta(x,\lambda;\mu,\rho)$ 
 satisfies
$$\cond\left(\nabla^2_{\lambda}\eta (x,\lambda;\mu,\rho)\right)=\mathcal{O}\!\left(1/{\mu}\right).$$
\end{theorem}
\begin{proof}
The proof can be found in Appendix~\ref{Appendix:proof of thm:condition number}.
\end{proof}

Subsequently, we define a merit function to determine the step length. The merit function $\delta(x,\lambda;\mu,\rho)$ is formulated using the second-order derivative of $\eta(x,\lambda;\mu,\rho)$, as defined below:
\begin{equation}\label{merit function 1}
	\delta(x,\lambda;\mu,\rho):=\sqrt{\dfrac{1}{\rho\mu}\Delta \lambda^{\top} \nabla_{\lambda}^2 \eta(x,\lambda;\mu,\rho)\Delta \lambda}.
\end{equation}

Let $\lambda(x;\mu,\rho):=\argmin_{\lambda}\eta(x,\lambda;\mu,\rho)$ and  $\widetilde{\Delta\lambda}:=\lambda-\lambda(x;\mu,\rho)$. Then we define
\begin{equation}\label{merit function 2} 
	\widetilde{\delta}(x,\lambda;\mu,\rho):=\sqrt{\dfrac{1}{\rho\mu}\widetilde{\Delta\lambda}^{\top}\nabla_{\lambda     }^2\eta(x,\lambda;\mu,\rho)\widetilde{\Delta\lambda}}. 
\end{equation}
Both $\delta(x,\lambda;\mu,\rho)$
and $\widetilde{\delta}(x,\lambda;\mu,\rho)$
play important roles in the complexity analysis of the algorithm.

\begin{remark}\label{delta-estimate}
	Notice that 
	\begin{equation}
		\delta(x,\lambda;\mu,\rho)=\sqrt{\dfrac{1}{\rho\mu}}\left\|\mathcal{A}z(x,\lambda;\mu,\rho)-\rho b\right\|_{\eta(x,\lambda;\mu,\rho)}^*,
	\end{equation}
	where $\left\|h\right\|_{\eta(x,\lambda;\mu,\rho)}^*     =\sqrt{h^{\top}(\nabla_{\lambda}^2 \eta(x,\lambda;\mu,\rho))^{-1}h}$, which has wide applications in path-following IPMs \cite{nesterov1998primal,nesterov1999infeasible,skajaa2015homogeneous}. If $\lambda=\lambda(x;\mu,\rho)$, then $\delta(x,\lambda;\mu,\rho)=\widetilde{\delta}(x,\lambda;\mu,\rho)=0.$ The set
	\begin{equation}\label{NeighborhoodofTrajectory}
		\begin{aligned}
			\mathcal{N}(\kappa)&=\left\{(x,\lambda,\mu,\rho) \mid \delta(x,\lambda;\mu,\rho)\leq \kappa\right\}\\
			&=\left\{(x,\lambda,\mu,\rho) \mid \left\|\mathcal{A}z(x,\lambda;\mu,\rho)-\rho b\right\|_{\eta(x,\lambda;\mu,\rho)}^*\leq{\kappa\sqrt{\rho\mu}}\right\}
		\end{aligned}
	\end{equation}
	is a neighborhood of the trajectory $\mathcal{N}(0)=\left\{(x,\lambda(x;\mu,\rho),\mu,\rho) \mid \mu>0,\rho>0\right\}$. Typically, $\kappa$ is set to $\frac{1}{4}$.
\end{remark}

We can now present the Newton method for solving the subproblem~\eqref{eq subproblem} in Algorithm~\ref{algsbal}.
\begin{algorithm}[H]
	\caption{ Newton method for \eqref{eq subproblem} at the $k$-th iteration in Algorithm~\ref{algsbal}}
	\begin{algorithmic}
		\State \textbf{Step 1: }  Input $x^{(k)}\in \J$, $\rho^{(k)}$, $\mu^{(k)}>0$ and initialize 
        $\lambda^{(k,0)}=\lambda^{(k)}$.  Let $j:=0$.
        \State \textbf{Step 2: }  Compute the search direction $\Delta\lambda^{(k,j)}$  using \eqref{Newton direction}.
        \State \textbf{Step 3: }  Let $\delta^{(k,j)}:=\delta(x^{(k)},\lambda^{(k,j)};\mu^{(k)},\rho^{(k)})$ and  $\hat{\kappa}: =  \dfrac{1}{\sqrt{\rho^{(k)} \mu^{(k)}} \Vert \lambda^{(k,j)} \Vert_2} $. 
        \State {\qquad\qquad \, }If $k=0$ and $\delta^{(k,j)}\le \kappa$, or  $\delta^{(k,j)}\leq \min \left\{\kappa, \hat{\kappa} \right\}$ for $k>0$, stop the  \State {\qquad\qquad \, }algorithm and return $\lambda^{(k+1)}=\lambda^{(k,j)}$.
        \State \textbf{Step 4: } Update $\lambda^{(k,j+1)}=\lambda^{(k,j)}+{\alpha}^{(k,j)}\Delta\lambda^{(k,j)} $, where $$ {\alpha}^{(k,j)}=
			\begin{cases}
				\quad\ \, 1, & \text{if $\delta^{(k,j)}< 2-\sqrt{3}$;} \\
				\dfrac{1}{1+\delta^{(k,j)}}, 
				& \text{if $\delta^{(k,j)}\geq 2-\sqrt{3}$.}
		\end{cases} $$ 
        \State \textbf{Step 5: }  Set $j:=j+1$ and go to Step 2.
  \end{algorithmic}\label{alg4}
\end{algorithm}

\section{Complexity analysis}\label{sec4}
This section establishes an $\mathcal{O}\left(1/{\epsilon}\right)$  iteration complexity of the \NAL for solving SCP.  We begin by drawing some results of the self-concordance property from \cite{nesterov1994interior} for the estimate of $\eta$.

The update rule
$
\rho^{(k+1)}=\max\{\tfrac12\,\rho^{(k)},\,\rho_{\min}\}
$ in Algorithm~\ref{algsbal}
implies that $\rho^{(k)}$ decreases geometrically until it reaches $\rho_{\min}$, and then stays unchanged. In particular, the number of outer iterations needed for $\rho^{(k)}$ to hit $\rho_{\min}$ is bounded by
\[
k_\rho \;\le\; \left\lceil \log_2\!\Bigl(\frac{\rho^{(0)}}{\rho_{\min}}\Bigr)\right\rceil .
\]
Hence, if Algorithm~\ref{algsbal} terminates with $\rho^{(k)}>\rho_{\min}$, the outer loop has executed at most a constant number of iterations. If it terminates with $\rho^{(k)}=\rho_{\min}$, then after at most $k_\rho$ outer iterations we enter the regime $\rho^{(k)}\equiv \rho_{\min}$, and the outer-loop complexity is determined entirely by the number of iterations in this regime. Therefore, for notational simplicity, we assume that $\rho^{(k)}$ remains unchanged and denote it simply by $\rho$.
\begin{lemma}\label{inequality 2025}
Given any $\mu>0$,  $\rho>0$, and $\Delta \lambda\in\mathbb{R}^m$, let $\delta :=\delta(x,\lambda;\mu,\rho)$ be defined in \eqref{merit function 1}. For any $\tau \in [0, 1]$ such that $\tau \delta < 1$, and for any $ h, h_1, h_2 \in \mathbb{R}^m$, the following inequalities hold:
\begin{subequations}
\begin{align}
\begin{aligned}\label{inequality 15}
(1-\tau\delta)^2h^{\top}\nabla^2_{\lambda}\eta(x,\lambda;\mu,\rho)h\leq h^{\top}\nabla^2_{\lambda}\eta(x,\lambda&+\tau\Delta \lambda;\mu,\rho)h\\
&\leq(1-\tau\delta)^{-2}h^{\top}\nabla_{\lambda}^2\eta(x,\lambda;\mu,\rho)h,\end{aligned}\\
\begin{aligned}\label{inequality 2}|h_1^{\top}[\nabla^2_{\lambda}\eta(x,\lambda&+\tau\Delta \lambda;\mu,\rho)-\nabla^2_{\lambda}\eta(x,\lambda;\mu,\rho)]h_2|\\&\leq[(1-\tau\delta)^{-2}-1]\sqrt{h_1^{\top}\nabla^2_{\lambda}\eta(x,\lambda;\mu,\rho)h_1}\sqrt{h_2^{\top}\nabla^2_{\lambda}\eta(x,\lambda;\mu,\rho)h_2}.\end{aligned}
\end{align}
\end{subequations}
\end{lemma}


Lemma~\ref{inequality 11} provides an estimate of the reduction in the merit function during each Newton iteration, as derived from \cite[Theorem 2.2.2, Theorem 2.2.3]{nesterov1994interior}. 
\begin{lemma}\label{inequality 11}
	Given any $\mu>0$ and $\rho>0$, let $\delta:=\delta(x,\lambda;\mu,\rho)$ and $\widetilde{\delta}:= \widetilde{\delta}(x,\lambda;\mu,\rho)$ be defined in \eqref{merit function 1} and \eqref{merit function 2}. The following inequalities hold: 
	\begin{enumerate}
		\item[(i)] If $\delta > 2-\sqrt{3}$, then
		\begin{equation}
			\eta(x,\lambda+\dfrac{1}{1+\delta}\Delta \lambda;\mu,\rho)\leq\eta(x,\lambda;\mu,\rho)-\alpha(\mu)\left(\delta-\ln(1+\delta)\right).
		\end{equation} 
		\item[(ii)] If $\delta\leq2-\sqrt{3}$, then
		\begin{equation}
			\delta(x,\lambda+\Delta\lambda;\mu,\rho)\leq \left(\dfrac{\delta}{1-\delta}\right)^2\leq\dfrac{\delta}{2}.
		\end{equation}
		\item[(iii)]If $\delta<\dfrac{1}{3}$, then $\widetilde{\delta}<1-\left(1-3\delta\right)^{\frac{1}{3}}$.
	\end{enumerate}
\end{lemma}

To estimate the iteration complexity of the Newton method for \eqref{eq subproblem}, we define the following function:
\begin{equation*}
	\theta(\mu^{(k)}):=\eta(x^{(k)},\lambda^{(k)};\mu^{(k)},\rho)-\eta(x^{(k)},\lambda(\mu^{(k)});\mu^{(k)},\rho),
\end{equation*}
where $\eta(x^{(k)},\lambda^{(k)};\mu^{(k)},\rho)$ represents the initial objective value at the $k$-th iteration, and $\eta(x^{(k)},\lambda(\mu^{(k)});\mu^{(k)},\rho)$ represents the optimal objective value at the $k$-th iteration. For simplicity, we omit the superscript and let 
\begin{equation}\label{definition of phi(mu)}
	\theta(\mu)=\eta(x,\lambda;\mu,\rho)-\eta(x,\lambda(\mu);\mu,\rho).
\end{equation}
Then $\theta(\mu)$ satisfies the following properties.

\begin{lemma}\label{inequality 10}
	For any $\rho>0$ and $\mu>0$, let $\widetilde{\delta}:= \widetilde{\delta}(x,\lambda;\mu,\rho)$ be defined in  \eqref{merit function 2}.
     If $\widetilde{\delta}<1$, then the following inequalities hold: 
	\begin{subequations}
		\begin{align}
			\theta(\mu)&\leq\left(\frac{\tilde{\delta}}{1-\tilde{\delta}}+\ln(1-\tilde{\delta})\right)\rho\mu,\label{inequality 3}\\
			\left|\theta'(\mu)\right|&\leq-\rho\sqrt{\nu}\ln(1-\widetilde{\delta}),\label{inequality 4}\\
			\left|\theta''(\mu)\right|&\leq\dfrac{2\rho\nu}{\mu}.\label{inequality 9}
		\end{align}
	\end{subequations}

\end{lemma}
\begin{proof}
	According to the definition of $\theta(\mu)$ in \eqref{definition of phi(mu)}, we obtain 
	\begin{equation}\label{definition2 of phi(mu)}
		\theta(\mu)=\int_0^1 \widetilde{\Delta\lambda}^{\top}\nabla_{\lambda}\eta(x,\lambda(\mu)+t\widetilde{\Delta\lambda};\mu,\rho)\ dt.
	\end{equation}
	Since $\lambda(\mu)=\argmin_{\lambda} \eta(x,\lambda;\mu,\rho)$, it follows that $\nabla_{\lambda}\eta(x,\lambda(\mu);\mu,\rho)=0$. Then \eqref{definition2 of phi(mu)} can be further written as 
	\begin{equation}\label{definition3 of phi(mu)}
		\theta(\mu)=\int_0^1\int_0^{t} \widetilde{\Delta\lambda}^{\top}\nabla_{\lambda}^2\eta(x,\lambda(\mu)+w\widetilde{\Delta\lambda};\mu,\rho)\widetilde{\Delta\lambda}\ dwdt.
 	\end{equation}
 Combining with \eqref{inequality 15}, we obtain
 \begin{equation}\label{inequality 6}
 	\widetilde{\Delta\lambda}^{\top}\nabla_{\lambda}^2\eta(x,\lambda(\mu)+w\widetilde{\Delta\lambda};\mu,\rho)\widetilde{\Delta\lambda}\leq \dfrac{\rho\mu\widetilde{\delta}^2}{(1-\widetilde{\delta}+w\widetilde{\delta})^2}.
 \end{equation}
From \eqref{definition3 of phi(mu)}, we can further deduce that 
\begin{equation}
	\begin{aligned}
	\theta(\mu)&\leq\int_0^1\int_0^{t}\dfrac{\rho\mu\widetilde{\delta}^2}{(1-\widetilde{\delta}+w\widetilde{\delta})^2}\ dwdt\\
	&=\left(\frac{\tilde{\delta}}{1-\tilde{\delta}}+\ln(1-\tilde{\delta})\right)\rho\mu.
	\end{aligned}
\end{equation}
Therefore, inequality \eqref{inequality 3} holds.

 By differentiating \eqref{definition of phi(mu)} and using $\nabla_{\lambda}\eta(x,\lambda(\mu);\mu,\rho)=0$, we obtain 
\begin{equation}\label{derivative phi(mu)}
	\begin{aligned}
	\theta'(\mu)&=\eta'(x,\lambda;\mu,\rho)-\eta'(x,\lambda(\mu);\mu,\rho)-\lambda'(\mu)^{\top}\nabla_{\lambda}\eta(x,\lambda(\mu);\mu,\rho)\\
	&=\eta'(x,\lambda;\mu,\rho)-\eta'(x,\lambda(\mu);\mu,\rho)\\
	&=\int_0^1 \widetilde{\Delta\lambda}^{\top}\nabla_{\lambda}\eta'(x,\lambda(\mu)+t\widetilde{\Delta\lambda};\mu,\rho) \ dt.
	\end{aligned}
\end{equation}
According to Theorem~\ref{inequality 14} and \eqref{inequality 6}, 
\begin{equation}
	\begin{aligned}
	\left|	\theta'(\mu)\right|&\leq\int_0^1\left| \widetilde{\Delta\lambda}^{\top}\nabla_{\lambda}\eta'(x,\lambda(\mu)+t\widetilde{\Delta\lambda};\mu,\rho)\right| \ dt\\
	&\overset{(\ref{inequality 14-1}b)}{\leq}\int_0^1\sqrt{\dfrac{\rho\nu}{\mu}}\sqrt{	\widetilde{\Delta\lambda}^{\top}\nabla_{\lambda}^2\eta(x,\lambda(\mu)+t\widetilde{\Delta\lambda};\mu,\rho)\widetilde{\Delta\lambda}} \ dt\\
	&\overset{\eqref{inequality 6}}{\leq} \int_0^1 \sqrt{\dfrac{\rho\nu}{\mu}} \dfrac{\sqrt{\rho\mu}\widetilde{\delta}}{(1-\widetilde{\delta}+t\widetilde{\delta})}\ dt\\
	&=-\rho\sqrt{\nu}\ln(1-\widetilde{\delta}).
	\end{aligned}
\end{equation}
Inequality \eqref{inequality 4} is thereby proven.

By differentiating \eqref{definition of phi(mu)} once more, we obtain 
\begin{equation}
	\theta''(\mu)=\eta''(x,\lambda;\mu,\rho)-\eta''(x,\lambda(\mu);\mu,\rho)-(\lambda'(\mu))^{\top}\nabla_{\lambda}\eta'(x,\lambda(\mu);\mu,\rho).
\end{equation}
Differentiating $\nabla_{\lambda}\eta(x,\lambda(\mu);\mu,\rho)=0$  with respect to $\mu$  yields
\begin{equation}
	\nabla_{\lambda}\eta'(x,\lambda(\mu);\mu,\rho)+\nabla_{\lambda}^2\eta(x,\lambda(\mu);\mu,\rho)\lambda'(\mu)=0,
\end{equation}
and thus $\lambda'(\mu)=-(\nabla_{\lambda}^2\eta(x,\lambda(\mu);\mu,\rho))^{-1}\nabla_{\lambda}\eta'(x,\lambda(\mu);\mu,\rho)$.

Based on Theorem~\ref{gradent and hessian of eta}, it follows that
	\begin{subequations}
	\begin{align}
		\nabla_\lambda \eta'(x,\lambda(\mu);\mu,\rho) &=  \mathcal{A}z'(x,\lambda(\mu);\mu,\rho)=-\rho \mathcal{A}H^{-1}\nabla_s\phi(s(x,\lambda(\mu);\mu,\rho)),\label{gradient of Ly2.1}\\
		\nabla_{\lambda}^2 \eta(x,\lambda(\mu);\mu,\rho) &=\mathcal{A}H^{-1}W\mathcal{A}^{*}.\label{hessian of Ly2.1}
	\end{align}
\end{subequations}
For convenience of description, we omit $x,\,\lambda(\mu),\,\mu$ and $\rho$ in the functions. Then 
\begin{equation}\label{inequality 7}
	\begin{aligned}
	-(\nabla_{\lambda}\eta')^{\top}\lambda'(\mu)&=(\nabla_{\lambda}\eta')^{\top}(\nabla_{\lambda}^2 \eta)^{-1}(\nabla_{\lambda}\eta')\\
	&=\rho^2 \langle \nabla_s\phi, H^{-1}\mathcal{A}^{*}(\mathcal{A}H^{-1}W\mathcal{A}^{*})^{-1} \mathcal{A}H^{-1}(\nabla_s\phi) \rangle\\
	&=\rho^2 \langle \nabla_s\phi, H^{-\frac{1}{2}}W^{-\frac{1}{2}}G^{\top}(GG^{\top})^{-1}GW^{-\frac{1}{2}}H^{-\frac{1}{2}}(\nabla_s\phi)\rangle\\
    &\leq \rho^2 \langle \nabla_s\phi, H^{-1}W^{-1}(\nabla_s\phi)\rangle\\
    &\leq \rho^2 \langle \nabla_s\phi, W^{-1}(\nabla_s\phi)\rangle\\
    &= \dfrac{\rho\nu}{\mu},
	\end{aligned}
\end{equation}
where $G = \mathcal{A}W^{\frac{1}{2}}H^{-\frac{1}{2}}$. Noting that
\begin{equation}
	\begin{aligned}
		\eta'(x,\lambda(\mu);\mu,\rho)&=\rho\phi+\langle\rho\mu\nabla_s\phi,s'\rangle+\langle\rho x,s'\rangle+\langle\mathcal{A}^{*}\lambda+s-c,s'\rangle\\
		&=\rho\phi+\langle\rho\mu\nabla_s\phi+\rho x+\mathcal{A}^{*}\lambda+s-c,s'\rangle\\
		&=\rho\phi,
	\end{aligned}
\end{equation}
 further differentiation with respect to $\mu$ yields
 \begin{equation}
 	\eta''(x,\lambda(\mu);\mu,\rho)=\rho\phi'=\rho \langle \nabla_s\phi, s'\rangle=-\rho^2 \langle \nabla_s\phi, H^{-1}\nabla_s\phi\rangle.
 \end{equation}
  Therefore, 
  \begin{equation}\label{inequality 8}
  	-\eta''(x,\lambda(\mu);\mu,\rho)=\rho^2 \langle \nabla_s\phi, H^{-1}\nabla_s\phi \rangle\leq\rho^2 \langle \nabla_s\phi, W^{-1}\nabla_s\phi \rangle=\dfrac{\rho\nu}{\mu}.
  \end{equation}
Similarly, it can be proven that $\eta''(x,\lambda;\mu,\rho)<0$. Combining equations \eqref{inequality 7} and \eqref{inequality 8}, we obtain \eqref{inequality 9}.
\end{proof}



\begin{remark}\label{inexact-lambda}
    In practice, at the $k$-th iteration, $\lambda^{(k+1)}$ is obtained when $\delta^{(k,j)}\le \kappa = \frac{1}{4}$, which implies
    $\widetilde{\delta}(x^{(k)},\lambda^{(k+1)}, \mu^{(k)},\rho) \le 1-(1-3\delta^{(k,j)})^{\frac{1}{3}} < \frac{1}{2}$ from Lemma~\ref{inequality 11}. 
    Consequently, Lemma~\ref{inequality 10} ensures that
    \begin{equation}\label{stopping criteria-exact}
        \eta (x^{(k)},\lambda^{(k+1)},\mu^{(k)},\rho) - \inf\limits_\lambda \eta(x^{(k)},\lambda,\mu^{(k)},\rho) <  \rho\mu^{(k)},
    \end{equation}
    once $\lambda^{(k+1)}$ has been computed in Step 2 of Algorithm~\ref{algsbal}. This is consistent with the stopping criterion (\ref{stopping criteria}), provided that $\epsilon^{(k)} = \sqrt{2\mu^{(k)}} = \sigma^{\frac{k}{2}} \sqrt{2\mu^{(0)}}$.
\end{remark}

The conclusion below presents the upper bound on the difference between the initial value and the optimal value of the function $\eta(x^{(k)},\cdot;\mu^{(k)},\rho)$ after $\mu$ changes.

\begin{lemma}\label{inequality 13}
	For any $\rho>0$ and $\mu>0$, let $\widetilde{\delta}\leq\widetilde{\beta}<1$ and $\mu^{+}=\sigma\mu$ with $\sigma\in(0,1)$. Then
	\begin{equation}\label{inequality 12 }
	{
		\eta(x,\lambda;\mu^+,\rho)-\eta(x,\lambda(\mu^+);\mu^+,\rho)\leq \mathcal{O}(1)\rho\nu\mu^+,}
	\end{equation}
where $\mathcal{O}(1)$ depends only on $\widetilde{\beta}$ and $\sigma$.
\end{lemma}
\begin{proof}
	Lemma~\ref{inequality 10} states that 
	\begin{equation*}\begin{aligned}
		\theta(\mu^{+})& =\theta(\mu)+(\mu^+-\mu)\theta'(\mu)+\int_{\mu^+}^\mu\int_t^\mu\theta''(w)\ dw dt  \\
		&\leq\left[\frac{\tilde{\delta}}{1-\tilde{\delta}}+\ln(1-\tilde{\delta})\right]\rho\mu+\ln(1-\tilde{\delta})\rho\sqrt{\nu}(\mu-\mu^+) +2\rho\nu\int_{\mu^+}^\mu\int_t^\mu w^{-1}\ dw dt \\
        &\leq\left[\frac{\tilde{\delta}}{1-\tilde{\delta}}+\ln(1-\tilde{\delta})\right]\rho\mu+2\rho\nu(\mu-\mu^++\mu^{+}\ln\sigma)\\
        &\leq\left[\frac{\tilde{\beta}}{1-\tilde{\beta}}+\ln(1-\tilde{\beta})\right]\dfrac{\rho}{\sigma}\mu^++2\rho\nu(\dfrac{1}{\sigma}-1+\ln\sigma)\mu^+.
\end{aligned}\end{equation*}
This completes the proof.
\end{proof}

Next, we prove that the stopping criterion in Algorithm~\ref{alg4} is well-defined. For notational convenience, we adopt Rockafellar's notations \cite[Section 8]{rockafellar1970convex} and define $0^+ \mathcal{C}$ as the recession cone of the closed convex set $\mathcal{C}$, and $(\Phi 0^+): \mathbb{R}^m \rightarrow (0,+\infty]$ as the recession function of a proper, closed, and convex function $\Phi: \mathbb{R}^m \rightarrow (0,+\infty]$. It then 
follows from \cite[Theorem 8.4, Theorem 8.6, and Theorem 8.7]{rockafellar1970convex} that 
$\Phi$ is level bounded if and only if 
\begin{equation}\label{level-bounded-implication}
    (\Phi 0^+)(\Delta \lambda) \le 0 \implies \Delta \lambda = 0.
\end{equation}
As an example, and for the subsequent proof, we now provide the recession function of $\phi$.
\begin{lemma}\label{phi-recession}
    Let $\phi: \K \rightarrow (0,+\infty]$ be a natural barrier of $\K$. Then 
    \begin{equation}
        (\phi 0^+)(\Delta s) = 
        \left\{\begin{aligned}
            & +\infty, & {\rm when}\ \Delta s\notin \K;\\
            & 0, & {\rm when} \ \Delta s \in \K.
        \end{aligned}  \right.
    \end{equation}
\end{lemma}
\begin{proof}
    For a given $s\in {\rm int}\, (\K)$, 
    \begin{equation}
        (\phi 0^+)(\Delta s) \stackrel{\text{(i)}}{=} \sup\limits_{\alpha>0} \dfrac{\phi(s+\alpha \Delta s)-\phi(s)}{\alpha} \stackrel{\text{(ii)}}{=} \lim\limits_{\alpha \rightarrow +\infty} \dfrac{\phi(s+\alpha \Delta s)-\phi(s)}{\alpha}.
    \end{equation}
    Here, identity~(i) follows from \cite[Theorem 8.5]{rockafellar1970convex} since the definition is independent of $s$, and identity~(ii) comes from the monotonicity of $\dfrac{\phi (s+ \alpha \Delta s) - \phi (s)}{\alpha}$ with respect to $\alpha$. It remains only to show the case $\Delta s \in \K$. Since $\phi$ is a natural barrier, it is also a $\nu$-$\LHSC$. Hence,
    \begin{equation}
        \dfrac{\phi(s+\alpha \Delta s)-\phi(s)}{\alpha} = \dfrac{\phi(\frac{s}{1+\alpha}+\frac{\alpha \Delta s}{1+\alpha})-\phi (s)}{\alpha} -\nu \dfrac{\ln (1+\alpha)}{\alpha}.
    \end{equation}
    If $\Delta s \in {\rm int}\, (\K)$, then  $\frac{s}{1+\alpha}+\frac{\alpha \Delta s}{1+\alpha} \rightarrow \Delta s \in {\rm int}\, (\K)$ as $\alpha \rightarrow +\infty$, and therefore
    \begin{equation}
        (\phi 0^+)(\Delta s) = \lim\limits_{\alpha \rightarrow +\infty} \left\{ \dfrac{\phi(\frac{s}{1+\alpha}+\frac{\alpha \Delta s}{1+\alpha})-\phi (s)}{\alpha} -\nu \dfrac{\ln (1+\alpha)}{\alpha} \right\} =0.
    \end{equation}
    If instead $\Delta s\in \K \backslash {\rm int}\, (\K)$, then by the closeness of $\phi 0^+$ \cite[Theorem 8.5]{rockafellar1970convex} and the fact that $\phi 0^+ (\Delta s)$ is nonnegative for all $\Delta s \in \K$, it follows that $(\phi 0^+)(\Delta s) = 0$. This completes the proof.
    
\end{proof}

The level boundedness of the subproblem function $\eta (x,\cdot;\mu,\rho)$ is then established.
\begin{proposition}\label{eta-level-bounded}
   For any fixed $x,\,\mu>0$, and $\rho>0$, $\eta (x,\cdot;\mu,\rho)$ is level bounded.
\end{proposition}
\begin{proof}
    
    By \cite[Theorem 3.3]{ye2004linear}, Problem ($\operatorname{D}_{\mu}$) has a unique solution, which implies that the function
    $$
        \Phi_1(\lambda,s): =   -\rho\langle b,\lambda\rangle + \rho\mu \phi(s) +    \mathcal{I}_{ \left\{ (\lambda,s)\mid \mathcal{A}^*\lambda +s -c=0 \right\}  }(\lambda,s)
    $$
    is level bounded, as shown in \cite[Theorem 8.7]{rockafellar1970convex}, where $\mathcal{I}_{ \left\{ (\lambda,s)\mid \mathcal{A}^*\lambda +s -c=0 \right\}  }$ denotes the indicator function of $\left\{ (\lambda,s)\mid \mathcal{A}^*\lambda +s -c=0 \right\}$. Consequently, the implication (\ref{level-bounded-implication}) indicates that for some $(s,\lambda)\in \mathcal{F}_{\operatorname{D_{\mu}}} $,
    \begin{equation}\label{Du-recession}
        \begin{aligned}
           & (\Phi_1 0^+)(\Delta \lambda, \Delta s) \stackrel{\text{(i)}}{=} -\rho\langle b, \Delta \lambda \rangle + \rho \mu (\phi0^+)(\Delta s)   \\
           & \qquad \qquad \qquad \qquad \qquad  \qquad + \mathcal{I}_{\left\{ (\Delta \lambda, \Delta s)\mid \mathcal{A}^* \Delta \lambda + \Delta s=0 \right\}} (\Delta \lambda, \Delta s) \le 0\\
           & \implies (\Delta \lambda, \Delta s) = 0,
        \end{aligned}
    \end{equation}
    where identity (i) is obtained by \cite[Theorem 8.5]{rockafellar1970convex}. Recall that 
    \begin{equation}
	L(x,\lambda,s;\mu,\rho)=-\rho \langle{b}, \lambda \rangle+\rho\mu\phi(s)+\rho\langle x,\mathcal{A}^{*}\lambda+s-c\rangle+\dfrac{1}{2}\left\|\mathcal{A}^{*}\lambda+s-c\right\|^2.
    \end{equation} 
    Thus, 
    \begin{equation}\label{L-recession}
    \begin{aligned}
             &\ (L(x,\cdot,\cdot;\mu,\rho)0^+)(\Delta \lambda, \Delta s)\\
             = &\ \lim\limits_{\alpha \rightarrow + \infty} \dfrac{L(x,\lambda+ \alpha \Delta \lambda,s+\alpha \Delta s;\mu,\rho)-L(x,\lambda,s;\mu,\rho)}{\alpha} \\
         =&\ -{{\rho}}\langle b, \Delta \lambda \rangle + \rho \mu (\phi0^+)(\Delta s) + \rho \langle x,
        \mathcal{A}^* \Delta \lambda + \Delta s \rangle \\
         &\  \qquad + \langle \mathcal{A}^* \Delta \lambda +\Delta s, \mathcal{A}^* \lambda +s-c\rangle+\dfrac{1}{2} \lim\limits_{\alpha\rightarrow+\infty} \alpha \Vert \mathcal{A}^*\Delta \lambda + \Delta s\Vert^2.\\
    \end{aligned}
    \end{equation}
    It follows from Lemma~\ref{phi-recession}, \eqref{Du-recession}, and \eqref{L-recession} that 
    \begin{equation}
        (L(x,\cdot,\cdot;\mu,\rho)0^+)(\Delta \lambda, \Delta s) \le 0 \implies (\Delta \lambda, \Delta s) = 0,
    \end{equation}
    which implies $L(x,\cdot,\cdot;\mu,\rho)$ is level bounded. Therefore, $\eta (x,\cdot;\mu,\rho)$ is also level bounded as $\eta (x,\cdot;\mu,\rho) = \inf\limits_{\lambda}L(x,\cdot,s;\mu,\rho) $.
\end{proof}

\begin{corollary}
    The stopping criterion 
    \begin{equation}\label{alg-stopping criteria}
        \delta^{(k,j)} \le \min \left\{ \kappa, \hat{\kappa}\right\}
    \end{equation}
    in Algorithm~\ref{alg4} can be achieved in finite steps at the $k$-th iteration.
\end{corollary}

\begin{proof}
    From  Lemma~\ref{inequality 11} and Lemma~\ref{inequality 10}, if the stopping criterion is never achieved at the $k$-th iteration, $\Vert \lambda^{(k,j)} \Vert \rightarrow \infty$ when $j\rightarrow +\infty$. However, when $j$ is sufficiently large, 
    \begin{equation}
        \eta (x^{(k)},\lambda^{(k,j)};\mu^{(k)},\rho) \le \inf\limits_\lambda \eta (x^{(k)}, \lambda;\mu^{(k)},\rho)+  \rho\mu^{(k)},
    \end{equation}
    which, combined with Proposition~\ref{eta-level-bounded}, leads to a contradiction.
\end{proof}

We now show that, under mild assumptions, the sequence $\left\{ (\lambda^{(k)}, s^{(k)}) \right\}$ generated by Algorithm~\ref{algsbal} is bounded.



\begin{theorem}\label{boundedness of lambda s}
    Suppose that the subproblem 
    at each iteration is executed with the stopping criterion \eqref{alg-stopping criteria}. Then the sequence $\left\{ (\lambda^{(k)}, s^{(k)}) \right\}$ generated by Algorithm~\ref{algsbal} is bounded, and any accumulation point of this sequence is an optimal solution to Problem $\operatorname{(D)}$.
\end{theorem}
\begin{proof}
   Since the Slater condition holds for both Problem ($\operatorname{P}$) and Problem ($\operatorname{D}$), and $\mathcal{A}$ is of full row rank, Robinson's constraint qualification \cite[Definition 2.86]{bonnans2013perturbation} holds, which implies that the optimal solution set of Problem ($\operatorname{D}$) is compact \cite[Theorem 3.9]{bonnans2013perturbation}. Next, we show that 
    $-\langle  b, \lambda^{(k)}\rangle$ is bounded from above. By the update of $x^{(k+1)}$,
    $$
    \begin{aligned}
         \langle c, x^{(k+1)} \rangle & = \langle -\rho (x^{(k+1)}-x^{(k)}) + \mathcal{A}^* \lambda^{(k+1)}+s^{(k+1)}, x^{(k+1)} \rangle  \\
         & =-\rho \langle x^{(k+1)}-x^{(k)},x^{(k+1)}\rangle \\
         & \quad \qquad+ \mu^{(k)} \nu + 
         \langle \lambda^{(k+1)},\mathcal{A}x^{(k+1)}-b \rangle + \langle b, \lambda^{(k+1)}\rangle\\
         & = -\rho \langle x^{(k+1)}-x^{(k)},x^{(k+1)}\rangle + \mu^{(k)} \nu\\
         & \quad \qquad  + 
         \langle \lambda^{(k+1)},\nabla_{\lambda}\eta (x^{(k)},\lambda^{(k+1) };\mu^{(k)},\rho) \rangle + \langle b, \lambda^{(k+1)}\rangle,\\
         & \le \rho \Vert x^{(k+1)}-x^{(k)} \Vert \Vert x^{(k+1)}\Vert + \mu^{(k)} \nu  \\
         & \quad \qquad + 
         \Vert \lambda^{(k+1)} \Vert_2 \Vert \nabla_{\lambda}\eta (x^{(k)},\lambda^{(k+1) };\mu^{(k)},\rho) \Vert_2 + \langle b, \lambda^{(k+1)}\rangle.\\
    \end{aligned}
    $$
    Then 
    \begin{equation}\label{b,lambda-bound}
       \begin{aligned}
            -\langle b ,\lambda^{(k+1)} \rangle & \le  -\langle c, x^{(k+1)} \rangle
             + \rho \Vert x^{(k+1)}-x^{(k)} \Vert \Vert x^{(k+1)}\Vert + \mu^{(k)} \nu  \\
         & \quad \qquad + 
         \Vert \lambda^{(k+1)} \Vert_2 \Vert \nabla_{\lambda}\eta (x^{(k)},\lambda^{(k+1) };\mu^{(k)},\rho) \Vert_2.\\
       \end{aligned}
    \end{equation}
    From Remark~\ref{delta-estimate}, 
    \begin{equation}
    \begin{aligned}
         &\ \delta^{(k,j)}(x^{(k)},\lambda^{(k+1)}, \mu^{(k)},\rho)\\
          = &\ \sqrt{\frac{1}{\rho \mu^{(k)}}}\Vert \mathcal{A}z^{(k+1)}-\rho b\Vert^*_{\eta(x^{(k)},\lambda^{(k+1)}, \mu^{(k)},\rho)}\\
          \ge &\ \sqrt{\frac{1}{\rho \mu^{(k)}}} \sqrt{\langle  \mathcal{A}z^{(k+1)}-\rho b, {(\mathcal{A} \mathcal{A}^*)^{-1}} \left( \mathcal{A}z^{(k+1)} - \rho b\right) \rangle} \\
          \ge &\ \sqrt{ \frac{ 1 }{\rho \mu^{(k)}\lambda_{\max} (\mathcal{A}\mathcal{A}^*)} }  \Vert \nabla_{\lambda} \eta (x^{(k)},\lambda^{(k+1) };\mu^{(k)},\rho) \Vert_2,
    \end{aligned}
    \end{equation}
    which with the stopping criterion 
        \begin{equation}
        \delta^{(k,j)}(x^{(k)},\lambda^{(k+1)}, \mu^{(k)},\rho) \le \min \left\{ \kappa, \frac{1}{\sqrt{\rho \mu^{(k)}} \Vert \lambda^{(k+1)}\Vert_2 }\right\}
    \end{equation}
    implies that 
    \begin{equation}\label{s,lambda-bound-1}
        \Vert \lambda^{(k+1)} \Vert_2 \Vert \nabla_{\lambda}\eta (x^{(k)},\lambda^{(k+1) };\mu^{(k)},\rho) \Vert_2 < +\infty.
    \end{equation}
    The convergence of $\left\{ x^{(k)} \right\}_{k=1}^\infty$ as obtained in Theorem~\ref{convergence-x} indicates that 
    \begin{equation}\label{s,lambda-bound-2}
        -\langle c, x^{(k+1)} \rangle
             + \rho \Vert x^{(k+1)}-x^{(k)} \Vert \Vert x^{(k+1)}\Vert < + \infty.
    \end{equation}
    Combining \eqref{s,lambda-bound-1}, \eqref{s,lambda-bound-2}, and \eqref{b,lambda-bound}, we then obtain the boundedness of $-\langle b, \lambda^{(k)}\rangle$. Since
    \begin{equation}
        \Vert\mathcal{A}^*\lambda^{(k)}+s^{(k)}-c\Vert \rightarrow 0 \quad \text{when}\ k \rightarrow +\infty,
    \end{equation}
     $s^{(k)} \in \K$, and the optimal solution set of Problem $(\operatorname{D})$ is compact, $\left\{ (\lambda^{(k)},s^{(k)}) \right\}$ is then bounded following from \cite[Corollary 8.3.3, Theorem 8.4, Theorem 8.7]{rockafellar1970convex}. The accumulation point is also optimal by Theorem~\ref{convergence-x}. 
\end{proof}

     Next, we show the required decrease in the function value for the subproblem is uniformly bounded before termination.  Let $x(\mu)$ be the unique solution to Problem $(\operatorname{P_{\mu}})$ from \eqref{cone problem with barrier}, and let ${\rm val}\, (\operatorname{P_{\mu}})$ be the optimal value of Problem $(\operatorname{P_{\mu}})$. Let 
    \begin{equation}
        \begin{array}{ll}
            C_1:=\sup\limits_{\mu\le \mu^{(0)}} \left\{\Vert \bar{x} \Vert,\Vert x(\mu) \Vert \right\},
            &C_2 : =\sup\limits_{k} \left\{ \Vert \lambda^{(k)} \Vert , \Vert s^{(k)} \Vert \right\},\\
            C_3: = \Vert x^{(0)} -  \bar{x} \Vert,
            &C_4: = 4 (C_2+C_3+ \frac{1}{1-\sqrt{\sigma}})(C_3 + \frac{1}{1-\sqrt{\sigma}}).
        \end{array}
    \end{equation}
    The existence of the constant $C_1$ is guaranteed by \cite[Theorem 3.3]{ye2004linear}, which shows that $x(\mu)$ is uniformly bounded for any $0 < \mu \le \mu^{(0)}$. Likewise, Theorem~\ref{boundedness of lambda s} ensures the existence of the constant $C_2$.

\begin{proposition}\label{boundedness of difference-final}
		At the $(k+1)$-th iteration $(k\ge 0)$, we have
		$$
        \begin{aligned}
           & \eta (x^{(k+1)},\lambda^{(k+1)};\mu^{(k+1)},\rho)-\inf\limits_{\lambda}\eta (x^{(k+1)},\lambda;\mu^{(k+1)},\rho)\\
            & \qquad \qquad \qquad \le \mathcal{O}(1) \rho \nu \mu^{(k+1)}+ \rho C_4+ 4\rho C_1^2+2\rho C_3^2+2\rho \dfrac{1}{(1-\sqrt{\sigma})^2}.
        \end{aligned}
$$ 
\end{proposition}
\begin{proof}
    First, we show that $\Vert s(x,\lambda;\sigma \mu, \rho)\Vert \le \Vert s(x,\lambda; \mu, \rho)\Vert$. Differentiating the function $\frac{1}{2}\Vert s(x,\lambda;\mu, \rho)\Vert^2$ with respect to $\mu$, we obtain 
    $$
    \begin{aligned}
          \frac{d}{d\mu}\left\{ \frac{1}{2}\Vert s(x,\lambda; \mu, \rho)\Vert^2  \right\} &  = - \rho \langle s, H^{-1} \nabla_s \phi (s) \rangle\\
         & = \rho \langle s, (I + \rho \mu \nabla_{s}^2 \phi (s) )^{-1} \nabla_{s}^2 \phi (s) s \rangle \\
         & \ge 0.
    \end{aligned}
    $$
    Thus, $\Vert s(x,\lambda;\sigma \mu, \rho)\Vert \le \Vert s(x,\lambda; \mu, \rho)\Vert$. Since $\eta (x,\lambda;\mu,\rho)$ is concave in $x$, it follows that
    \begin{equation}\label{thm-final-inq1}
    \begin{array}{ll}
     \eta (x^{(k+1)},\lambda^{(k+1)};\mu^{(k+1)},\rho)  \le \eta (x^{(k)},\lambda^{(k+1)};\mu^{(k+1)},\rho)\\
     \qquad \qquad \qquad \qquad + \langle \nabla_{x} \eta (x^{(k)},\lambda^{(k+1)};\mu^{(k+1)},\rho), x^{(k+1)}-x^{(k)}\rangle. \\
    \end{array}
    \end{equation}
    By \eqref{gradient of eta with x} and the monotonicity of $\Vert s(x^{(k)},\lambda^{(k+1)};\cdot,\rho)\Vert$, 
    \begin{equation}\label{thm-final-inq2}
    \begin{aligned}
         &\ \Vert \nabla_{x} \eta (x^{(k)},\lambda^{(k+1)};\mu^{(k+1)},\rho)\Vert \\
         =&\
         \rho \Vert \mathcal{A}^* \lambda^{(k+1)} + s (x^{(k)},\lambda^{(k+1)};\sigma \mu^{(k)},\rho)-c\Vert\\
          \le&\ \rho \Vert  s (x^{(k)},\lambda^{(k+1)}; \mu^{(k)},\rho) -  s (x^{(k)},\lambda^{(k+1)}; \sigma \mu^{(k)},\rho)\Vert + \rho \Vert x^{(k+1)}-x^{(k)}\Vert \\
          \le &\ 2\rho C_2 + \rho \Vert x^{(k+1)}-x^{(k)} \Vert.
    \end{aligned}
    \end{equation}
    It follows from the properties of maximal monotone operator (see \cite[Theorem 1]{rockafellar1976monotone}), together with \cite[Proposition 6]{rockafellar1976augmented} and the stopping criterion (\ref{stopping criteria-exact}), that
    \begin{equation}\label{thm-final-inq3}
        \Vert x^{(k+1)}-x^{(k)}\Vert \le 2\Vert x^{(k)}-\bar{x} \Vert+ \sqrt{\mu^{(k)}}
        \le 2C_3 + \frac{2}{1-\sqrt{\sigma}}.
    \end{equation}
    Combining (\ref{thm-final-inq1}), (\ref{thm-final-inq2}) and (\ref{thm-final-inq3}),
    \begin{equation}\label{thm-final-inq4}
    \begin{aligned}
        &\ \eta (x^{(k+1)},\lambda^{(k+1)};\mu^{(k+1)},\rho)\\
         \le &\ \eta (x^{(k)},\lambda^{(k+1)};\mu^{(k+1)},\rho) + 4\rho \left( C_2+C_3+ \frac{1}{1-\sqrt{\sigma}}\right) \left(C_3 + \frac{1}{1-\sqrt{\sigma}}\right) \\
          \stackrel{\text{(i)}}{\le} &\  \inf\limits_{\lambda}\eta (x^{(k)},\lambda;\mu^{(k+1)},\rho) + \mathcal{O}(1) \rho \nu \mu^{(k+1)}+ \rho C_4\\
          \stackrel{\text{(ii)}}{\le} &\ -{\rm val}\, (\operatorname{P}_{\mu^{(k+1)}}) + \mathcal{O}(1) \rho \nu \mu^{(k+1)}+ \rho C_4,
    \end{aligned}
    \end{equation}
    where (i) follows from Lemma~\ref{inequality 13}, and (ii) comes from (\ref{augmented Lagrangian trans}) and the definition of proximal mapping. Moreover, by (\ref{augmented Lagrangian trans}) and the definition of proximal mapping, 
    \begin{equation}\label{thm-final-inq5}
    \begin{aligned}
        &\ \inf\limits_{\lambda}\eta (x^{(k+1)},\lambda;\mu^{(k+1)},\rho)   \\
         \ge &\ -{\rm val}\, (\operatorname{P}_{\mu^{(k+1)}})- \frac{\rho}
         {2}\Vert x(\mu^{(k+1)})-x^{(k+1)}\Vert^2\\
         \ge &\ -{\rm val}\, (\operatorname{P}_{\mu^{(k+1)}})- \rho \Vert x(\mu^{(k+1)})-\bar{x}\Vert^2 - \rho \Vert \bar{x} 
            - x^{(k+1)}\Vert^2 \\
        \stackrel{\text{(iii)}}{\ge} &\ -{\rm val}\, (\operatorname{P}_{\mu^{(k+1)}})- \rho \Vert x(\mu^{(k+1)})-\bar{x}\Vert^2 - \rho \left(\Vert \bar{x} 
            - x^{(0)}\Vert+ \frac{1}{1-\sqrt{\sigma}}\right)^2  \\
         \ge &\ -{\rm val}\, (\operatorname{P}_{\mu^{(k+1)}})- 4\rho C_1^2-2\rho C_3^2-2\rho \frac{1}{(1-\sqrt{\sigma})^2},
    \end{aligned}
    \end{equation}
    where (iii) is obtained similar to that in (\ref{thm-final-inq3}). Consequently, by (\ref{thm-final-inq4}) and (\ref{thm-final-inq5}), 
    \begin{equation}\label{thm-final-inq6}
        \begin{array}{ll}
             \eta (x^{k+1)},\lambda^{(k+1)};\mu^{(k+1)},\rho)-\inf\limits_{\lambda}\eta (x^{(k+1)},\lambda;\mu^{(k+1)},\rho) \\
             \qquad \qquad \qquad \le \mathcal{O}(1) \rho \nu \mu^{(k+1)}+ \rho C_4+ 4\rho C_1^2+2\rho C_3^2+2\rho \dfrac{1}{(1-\sqrt{\sigma})^2}.
        \end{array}
    \end{equation}
    Thus, the conclusion is proven.
\end{proof}

Based on the above conclusions, we can present the most important result of this paper, which demonstrates that the \NAL has the $\mathcal{O}(1/\epsilon)$ complexity bound.

\begin{theorem}\label{iteration complexity}
		Algorithm~\ref{algsbal} requires at most  $\mathcal{O}\left( 1/\epsilon\right)$ iterations to attain the desired accuracy $\epsilon$.
\end{theorem}
\begin{proof}
	First, we estimate the number of iterations for the outer loop, denoted as $N_{\text{out}}$. Assume that at the $k$-th iteration, $\mu^{(k)}\leq\epsilon$ holds. Notice that 
	$
	\mu^{(k)}=\sigma \mu^{(k-1)}=\cdots=\sigma^k \mu^{(0)},
	$ and therefore, 
	$
	k\geq \ln(\frac{\mu^{(0)}}{\epsilon})/\ln(\frac{1}\sigma).
	$
	 This implies that $N_{\text{out}}\leq \ln(\frac{\mu^{(0)}}{\epsilon})/\ln(\frac{1}\sigma)+1.$

	Then we estimate the number of iterations for the $k$-th inner loop, denoted as $N_{\text{in}}^{(k)}$.
    At the $k$-th iteration,  recall that when the stopping criterion is obtained,
    $$
        \hat{\kappa} =  \frac{1}{\sqrt{\rho \mu^{(k)}} \Vert \lambda^{(k)} \Vert_2} \ge \dfrac{1}{\sqrt{\rho \mu^{(k)}}C_2}.
    $$
    If  $\delta^{(k,j)} \geq \hat{\kappa} \ge \kappa $ for  $j=1,\dots,l$, then
$$\delta^{(k,j)}-\ln(1+\delta^{(k,j)})\geq  \kappa-\ln(1+\kappa).$$
 According to Lemma~\ref{inequality 11}(i), since $\kappa - \ln (1+\kappa) \ge 0.02$, we have	
 \begin{equation}
 	\begin{aligned}
    \inf\limits_\lambda \eta(x^{(k)},\lambda;\mu^{(k)},\rho)&\leq\eta(x^{(k)},\lambda^{(k,l)};\mu^{(k)},\rho)\\&\leq\eta(x^{(k)},\lambda^{(k,0)};\mu^{(k)},\rho) -0.02\rho \sigma^{k}\mu^{(0)} l.
 	\end{aligned}
 \end{equation}
Therefore, from Proposition~\ref{boundedness of difference-final}, after
$$
l=\left\lfloor\dfrac{\mathcal{O}(1) \rho \nu \mu^{(k+1)}+ \rho C_4+ 4\rho C_1^2+2\rho C_3^2+2\rho \frac{1}{(1-\sqrt{\sigma})^2}}{\rho \mu^{(0)}} \dfrac{50}{\sigma^{k}}\right\rfloor
$$ 
iterations, $\delta^{(k,l)}$ $\leq2-\sqrt{3}$. Lemma~\ref{inequality 11}(ii) states that following one more Newton iteration, $\delta^{(k,l+1)}\leq\frac{2-\sqrt{3}}{2}\le \kappa = \min \{ \kappa,\hat{\kappa} \}$ is satisfied. Let 
\begin{equation}
    C_5: = \mathcal{O}(1) \rho \nu \mu^{(k+1)}+ \rho C_4+ 4\rho C_1^2+2\rho C_3^2+2\rho \frac{1}{(1-\sqrt{\sigma})^2}.
\end{equation}
If $\hat{\kappa} \le \kappa \le \delta^{(k,j)}$, similar to the preceding discussions, after at most 
\begin{equation}
    l=\left\lfloor\dfrac{C_5}{\rho \mu^{(0)}} \dfrac{50}{\sigma^{k}}\right\rfloor+1
\end{equation}
iterations, we have $\delta^{(k,l)} \le \kappa$. It then follows from Lemma~\ref{inequality 11}(ii) that after at most
\begin{equation}
    l = \left\lfloor\dfrac{C_5}{\rho \mu^{(0)}} \dfrac{50}{\sigma^{k}}\right\rfloor+1+ \left\lfloor \dfrac{\ln \left( \kappa \sqrt{\rho \mu^{(k)}}C_2 \right)}{\ln 2} \right\rfloor
\end{equation}
iterations, $\delta^{(k,j)} \le \hat{\kappa} = \min\{ \kappa, \hat{\kappa} \}$. Let 
$$
\begin{aligned}
  C_{6}
  := &\left\lfloor\dfrac{C_5}{\rho \mu^{(0)}} \right\rfloor+ \left\lfloor \dfrac{\ln \left( \kappa \sqrt{\rho \mu^{(0)}}C_2 \right)}{\ln 2} \right\rfloor+2.
\end{aligned}
$$
Consequently, at the $k$-th iteration, we need at most $N_{\text{in}}^{(k)}:= \left\lfloor \dfrac{C_{6}}{\sigma^k} \right\rfloor$ iterations to stop. The total number of iterations $N$ is 
\begin{equation}
    N = \sum_{k=1}^{N_{\text{out}}}N_{\text{in}}^{(k)}  = \sum_{k=1}^{\mathcal{O}(\ln(\frac{1}{\epsilon}))} \left\lfloor \dfrac{C_{6}}{\sigma^k} \right\rfloor=\mathcal{O}\left(1/{\epsilon} \right).
\end{equation}
This completes the proof.
\end{proof}

\section{Computational results}\label{sec5}

In this section, we evaluate the efficiency of Algorithm~\ref{algsbal} by testing it on three classical SCP problems: SDP, SOCP, and LP. For the SDP problems, we test 43 instances from the SDPLIB\footnote{\url{https://github.com/vsdp/SDPLIB/tree/master/data/}} benchmark, which include SDP relaxations of graph partitioning problems (GPP), max-cut problems (MCP), box-constrained quadratic programming problems (QP), and  Lov$\acute{\text{a}}$sz theta problems (Theta). Regarding the SOCP problems, we focus on two categories with 28 instances:  minimum enclosing ball (MEB) problems and square-root Lasso problems. We also test 114 LP instances from the Netlib\footnote{\url{https://netlib.org/lp/data/}} benchmark. The results are compared with those from the competitive solvers ABIP \cite{lin2021admm}, SDPNAL+ \cite{yang2015sdpnal+}, SeDuMi \cite{sturm1999using}, Clarabel \cite{goulart2024clarabel} and Hypatia \cite{coey2023performance}. Detailed numerical results are provided in Section~\ref{Additional computational results}.

 All the computational results presented in this paper were obtained on a Windows 10 personal computer equipped with an Intel i5-8300H processor (4 cores, 8 threads, 2.3 GHz) and 16 GB of RAM. The \NAL for the tested problems was implemented in MATLAB 2022a. In the experiment, we set the initial barrier parameter $\mu^{(0)} = 0.1$ and the penalty parameter $\rho^{(0)} = 1$. 
 

An approximate solution is declared when the following condition is satisfied:
$$
    \max \left\{ {\rm pinfeas}, {\rm dinfeas}, \mu \right\} \le {\rm tol},
$$
where the relative primal and dual infeasibilities are defined as
$$
    {\rm pinfeas}: = \frac{\Vert \mathcal{A}z-\rho b \Vert}{1+\Vert b \Vert}, \ {\rm dinfeas}:= \frac{\Vert \mathcal{A}^* \lambda + s -c \Vert}{1+\Vert c \Vert}.
$$
By default, we set the maximum number of iterations to 100 and ${\rm tol}=10^{-6}$.

\subsection{Condition number}
\begin{figure}[h]
\centering
\includegraphics[scale = 0.65]{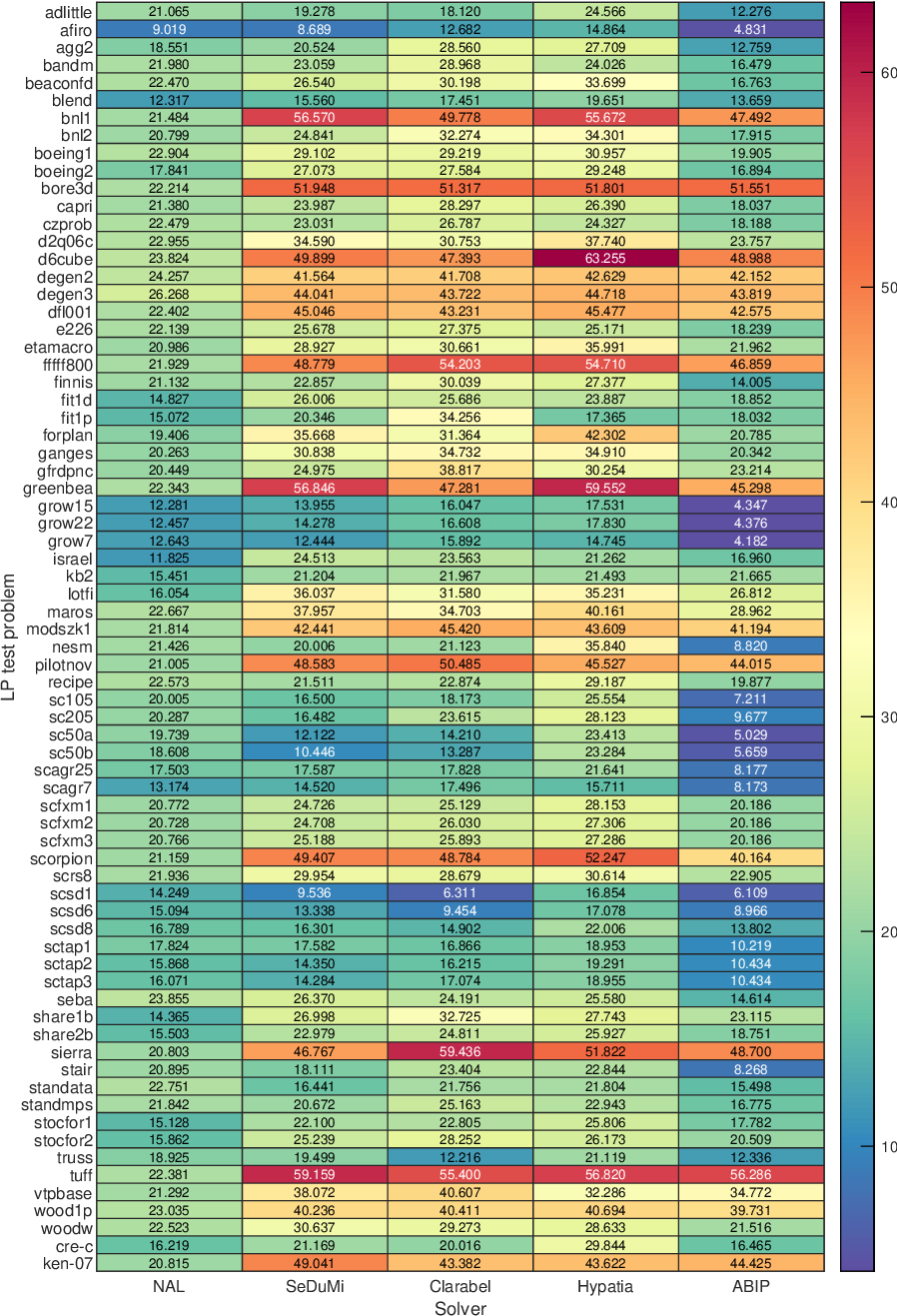}
\caption{\label{figure-condition} Logarithmic heatmap of condition numbers for NAL, SeDuMi, Clarabel, Hypatia, and ABIP.}
\end{figure}

Like classical IPMs, the \NAL solves a similar linear system \eqref{Newton direction}, which is also the most time-consuming part of the computation \cite{nocedal1999numerical}. If the condition number of a linear system is too large, the accuracy and stability of the solution may be significantly compromised, thus affecting the overall efficiency of the algorithm \cite{higham2002accuracy}. When applied to the LP problems, the diagonal elements in the \NAL are constrained within the range $(0,1)$, compared with the range $(0,\infty)$  in primal-dual IPMs \cite{nocedal1999numerical}. Therefore, the SCMs arising from the \NAL are theoretically expected to have smaller condition numbers. The following figure illustrates this point.

Fig. \ref{figure-condition} presents the condition numbers  of the SCMs from the subproblems solved by the \NALM, SeDuMi, Clarabel, Hypatia, and ABIP, corresponding to the LP problems.  Notably, the condition numbers for ABIP remain constant because the SCMs are $\mathcal{A}\mathcal{A}^*$ for all problems, while the condition numbers of the other solvers vary at each iteration. Based on this, we compute the logarithm of the geometric mean of the condition numbers across the iterations for each solver and problem, and visualize the results in a heatmap, as shown in Fig. \ref{figure-condition}.

From the comparison among the condition numbers for the \NALM, SeDuMi, Clarabel, and Hypatia, the \NAL consistently demonstrates a predominantly darker blue color, indicating that its condition numbers are relatively better. Specifically, for nearly 81\% of the test problems, the SCMs in the \NAL exhibit the smallest geometric means of the condition numbers. This observation is consistent with Remark~\ref{condition number remark}.

\subsection{Performance profile and shifted geometric mean}
Since some solvers fail to solve certain instances, we adopt the shifted geometric means (SGMs) of CPU time to ensure a fair comparison of solver performance on the benchmarks. We also employ performance profiles \cite{dolan2002benchmarking} to compare the numerical performance of different solvers. 

Specifically, the SGM is defined by
\begin{equation}\label{def-SGM}
    {\rm SGM} = \left(\prod\limits_{i=1}^n (t_i + {\rm sh})\right)^{1/n} - {\rm sh},
\end{equation}
where $t_i$ is the runtime for the $i$-th instance, $n$ is the total number of instances, and ${\rm sh}$ is the time-shift value to alleviate the effect of runtimes that are almost $0$. If the solver fails to solve the $i$-th instance, we set $t_i$ to the maximum time limit, MAXTIME. Specifically, for LP, MAXTIME is set to $3600$ seconds; for SOCP, it is set to $7200$ seconds; and for SDP, it is set to $43200$ seconds. Additionally, the parameter 
sh is set to $100$ for SDP, $10$ for SOCP, and $1$ for LP. We normalize all the SGMs by setting the smallest SGM to be 1.

Let $\mathbb{S}$ and $\mathbb{P}$ denote the set of solvers and problems, respectively. $t_{s,p}$ represents the CPU time taken by solver $s\in \mathbb{S}$ to solve problem $p\in \mathbb{P}$. Let $t_{s,p} = + \infty$ if solver $s$ fails to solve problem $p$. Then the performance ratio is defined by
$$
    r_{s,p} = \frac{t_{s,p}}{\min\limits_{s\in \mathbb{S}} \left\{ t_{s,p} \right\}}.
$$
The likelihood that solver $s\in \mathbb{S}$ achieves a performance ratio $r_{s,p}$ within a factor $\tau \in \mathbb{R}$ of the best performance is denoted by
$$
    \rho_{s} (\tau) = \frac{1}{| \mathbb{P} |} {\rm size}\left\{ p\in \mathbb{P}\ :\ r_{s,p}\le \tau \right\}.
$$
The performance profile displays the distribution function of the performance ratios for each solver. From this, we conclude that the solver with the highest value of $\rho_{s} (\tau)$ represents the best solver.

\subsubsection{Semidefinite programming}
Firstly, we test four important problem classes (GPP, MCP, QP, and Theta) from the SDPLIB benchmark. Table~\ref{SDP-solve} reports the number of problems successfully solved to an accuracy of $10^{-6}$ by different solvers along with their SGMs. Since ABIP does not support solving SDP problems currently, it is not included in this test. The results show that the \NAL solves all the problems and achieves the smallest SGM of solving time, indicating not only favorable performance on individual problems but also consistent numerical robustness across the benchmark.

\begin{table}[htbp]
\footnotesize
        \centering
        \caption{\label{SDP-solve} Percentage of SDP problems solved to the accuracy of $10^{-6}$ and SGM results.}
	\begin{tabular}{c|c c c c c}
		\toprule
		Solver name & \NALMM & SDPNAL+ & SeDuMi & Clarabel & Hypatia\\
        \hline
        Problems solved & \textbf{100\%} & 84\% & 98\% & 53\% & 81\%\\
        SGM &  \textbf{1.000} & 7.100 & 2.273 & 84.190 & 8.668\\
        \bottomrule
	\end{tabular}  
\end{table}


Fig. \ref{figure-sdp} presents the performance profile of the five solvers, comparing the CPU time required to solve the SDP problems. In this scenario, the \NAL solves nearly 60\% of the problems with the shortest CPU time, significantly outperforming all other solvers. It is worth noting that although SeDuMi successfully solves nearly all of the problems in this set and ranks as the second fastest solver, this performance is also dependent on the problem size. According to \cite{yang2015sdpnal+,zhao2010newton}, SDPNAL+ exhibits exceptionally fast solving speeds for specially structured large problems with $\mathcal{A}$ having up to $125000$ rows. This is consistent with our numerical experiments.

\begin{figure}[htbp]
\footnotesize
\centering
\includegraphics[scale = 0.5]{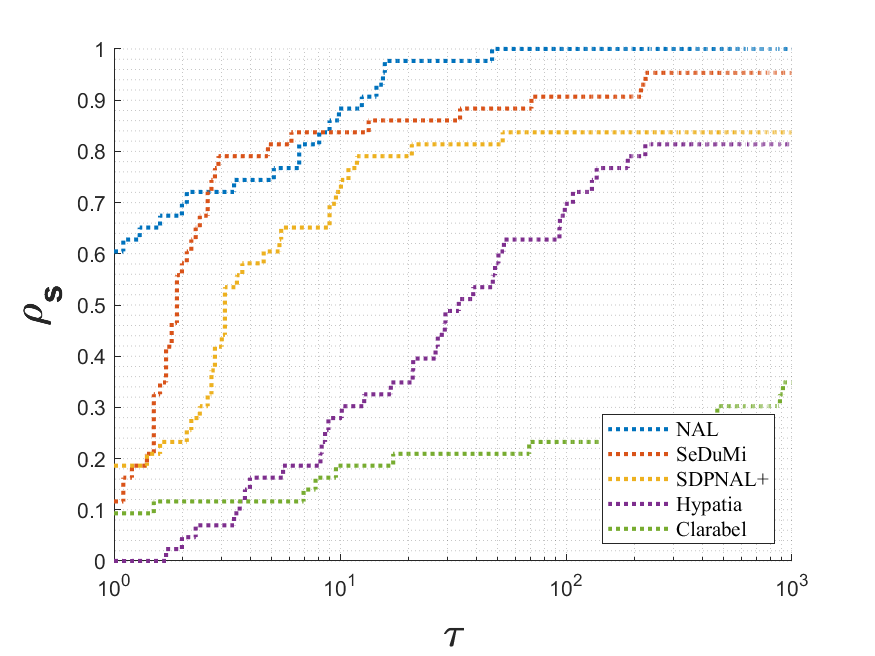}
\caption{\label{figure-sdp}{Performance profile on the SDP problems from SDPLIB with respect to CPU time.}}
\end{figure}

\subsubsection{Second-order cone programming}
Next, we test the \NAL on a set of SOCP problems that are taken from two distinct classes. The first is the MEB problem, which aims to compute the smallest radius ball that encompasses a given set of balls (including points). MEB is a fundamental problem in computational geometry with wide-ranging applications, including machine learning and data mining. For an introduction to MEB problems, we refer the reader to \cite{zhou2005efficient}. The second class is the square-root Lasso problem, which has been used extensively in high-dimensional statistics and machine learning. Both classes can be reformulated into SOCP problems, as detailed in \cite{lobo1998applications}.

We then randomly generate data at various dimensions and test both categories of these problems. Since SDPNAL+ is specifically designed for solving SDP problems, it is excluded from this experiment. Table~\ref{SOCP-solve} presents the SGMs for each solver and the \NALM, with the numbers of problems successfully solved to an accuracy of $10^{-6}$. Only the \NALM, SeDuMi, and Clarabel are able to solve all the problems. {Notably, the \NAL achieves the lowest SGM of solving time across all tested SOCP problems. Such behaviors underscore the numerical robustness and overall stability of the \NALM.}
\begin{table}[htbp]
\footnotesize
        \centering
        \caption{\label{SOCP-solve} Percentage of SOCP problems solved to the accuracy of $10^{-6}$ and SGM results.}
	\begin{tabular}{c|c c c c c}
		\toprule
		Solver name & \NALMM & ABIP & SeDuMi & Clarabel & Hypatia\\
        \hline
        Problems solved & \textbf{100\%} & 68\% & \textbf{100\%} & \textbf{100\%} & 54\%\\
        SGM &  \textbf{1.000} & 32.470 & 8.081 & 8.708 & 166.558\\
        \bottomrule
	\end{tabular}  
\end{table}



Fig. \ref{figure-socp} presents the performance profile of the four solvers compared to the \NAL for the MEB and square-root Lasso problems. It is evident that the \NAL solves at least 70\% of the problems in the shortest CPU time. Additionally, it completes all instances within a reasonable timeframe.

\begin{figure}[htbp]
\footnotesize
\centering

\includegraphics[scale = 0.5]{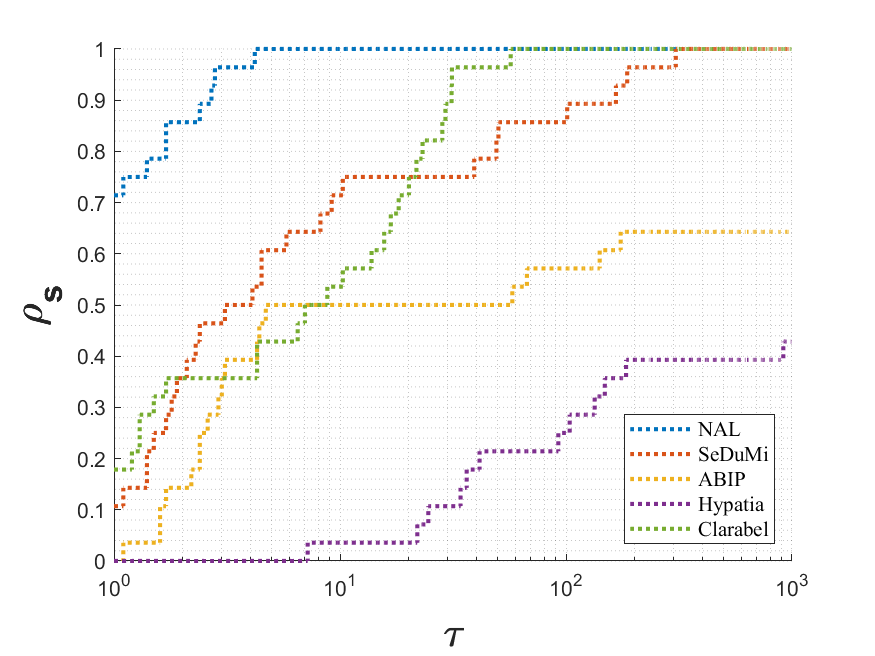}
\caption{\label{figure-socp}Performance profile on the SOCP problems with respect to CPU time.}
\end{figure}

\subsubsection{Linear programming}
Finally, we test the \NAL on the LP problems from the Netlib benchmark. Table~\ref{LP-solve} shows the total number and the SGMs of problems that have been successfully solved to the accuracy of $10^{-6}$ by the five solvers under their default settings. It can be observed that only the \NAL is able to solve all the instances to the accuracy of $10^{-6}$.  The \NAL achieves performance comparable to Clarabel on these problems, while outperforming the other solvers in most cases. 
\begin{table}[htbp]
\footnotesize
        \centering
        \caption{\label{LP-solve} Percentage of LP problems solved to the accuracy of $10^{-6}$ and SGM results.}
	\begin{tabular}{c|c c c c c}
		\toprule
		Solver name & \NALMM & ABIP & SeDuMi & Clarabel & Hypatia\\
        \hline
        Problems solved & \textbf{100\%} & 82\% & 96\% & 98\% & 72\%\\
        SGM &  \textbf{1.000} & 34.760 & 2.444 & 1.085 & 57.912\\
        \bottomrule
	\end{tabular}  
\end{table}








Fig. \ref{figure-lp} presents the performance profile of the five solvers. The CPU time required to solve 114 LP problems from the Netlib benchmark is compared. In this scenario, Clarabel requires the least amount of time on almost $68\%$ of the problems, outperforming other solvers. However, it performs less efficiently than the \NAL in the SGM statistics due to two main factors: first, Clarabel does not solve all the problems; second, it demonstrates excessively long CPU time for certain problems. This, in turn, highlights the stability of the \NALM. Furthermore, it is worth noting that Clarabel is implemented on the Julia platform, whereas the \NAL operates on the MATLAB platform, and the difference in performance could be attributed to the platform-specific factors.


\begin{figure}[htbp]
\footnotesize
\centering
\includegraphics[scale = 0.5]{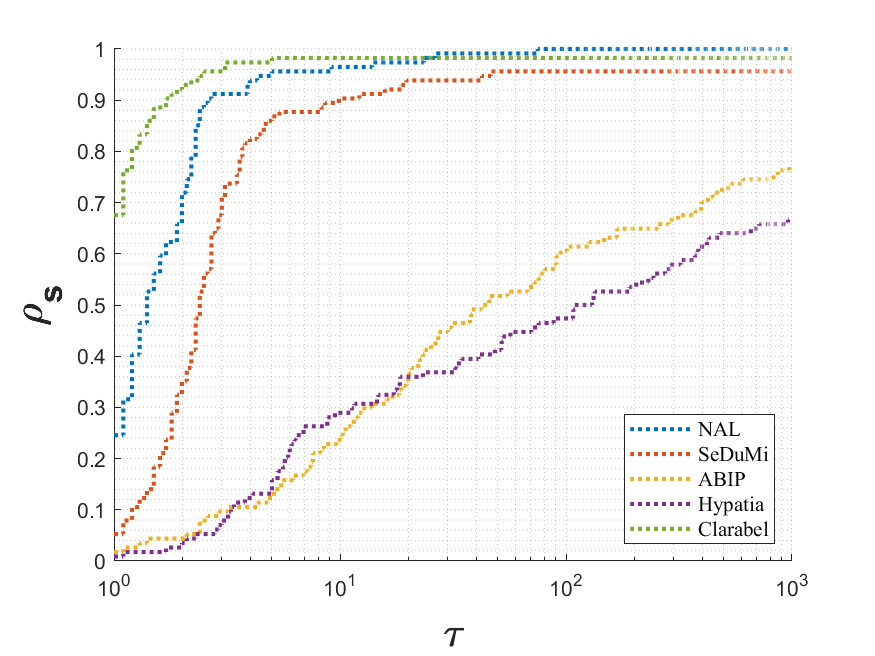}
\caption{\label{figure-lp}{Performance profile on the LP problems from Netlib with respect to CPU time.}}
\end{figure}

\section{Conclusion}\label{sec6}
 Inspired by the advantages of the ALM and IPMs, we have proposed an \NAL for solving SCP problems. The self-concordance of the novel augmented Lagrangian function has been proven, laying the foundation for the $\mathcal{O}(1/\epsilon)$ complexity bound of the \NALM. An easy-to-implement stopping criterion has also been obtained. In addition, we have shown that the condition numbers of the SCMs in the \NAL scale as $\mathcal{O}(1/\mu)$, which is more favorable than the $\mathcal{O}(1/\mu^{2})$ behavior of classical IPMs, as further evidenced by a detailed condition number heatmap. Numerical experiments on standard benchmarks have confirmed the improved performance, enhanced stability, and computational efficiency of the proposed method, demonstrating its advantages over several existing methods.  

\appendix
\section{Proof of Theorem~\ref{thm:condition number}}\label{Appendix:proof of thm:condition number}

\begin{proof}
By Theorem~\ref{convergence-x} and Theorem~\ref{boundedness of lambda s}, the sequences $\{x^{(k)}\}$, $\{s^{(k)}\}$, and $\{\lambda^{(k)}\}$ are bounded. Since $\rho^{(k)}$ is set to be bounded away from $0$, there exists a constant $\varrho>0$ such that $0<\varrho \le \rho^{(k)} \le \rho^{(0)}$ for all $k\ge 0$. It then follows from equation \eqref{Eq:function z} that $\{z^{(k)}\}$ is bounded as well. Hence, there exists a constant $U>0$ such that
\[
\max\limits_{k} \{\|s^{(k)}\|,\|z^{(k)}\|\}\le U.
\]

For notational simplicity, in the proofs below, we omit the superscripts $(k)$ indicating the iterates and denote $\nabla^2_{\lambda} \eta (x,\lambda;\mu,\rho)$ by $\nabla^2_\lambda \eta $.
 A direct computation shows that $\lambda_{\max}(\nabla_{\lambda }^2 \eta)\le \lambda_{\max}(\mathcal{A}\mathcal{A}^*)$. We next derive a lower spectral bound for $\nabla_{\lambda }^2 \eta$.

Let $u:=\rho x-c+\mathcal{A}^{*}\lambda\in\J$. Consider its spectral decomposition $u=\sum_{i=1}^{r}\lambda_i(u)\,v_i,$ where $\{v_i\}_{i=1}^{r}$ is a Jordan frame.  According to equations \eqref{Eq:function s} and \eqref{Eq:function z}, the associated variables $s$ and $z$ admit the same Jordan frame:
\[
s=\sum_{i=1}^{r}\lambda_i(s)\,v_i,
\qquad
z=\sum_{i=1}^{r}\lambda_i(z)\,v_i,
\]
with eigenvalues given by
$\lambda_i(s)=\frac{(\lambda_i(u)^2+4\rho\mu)^{\frac{1}{2}}-\lambda_i(u)}{2}$,
$\lambda_i(z)=\frac{(\lambda_i(u)^2+4\rho\mu)^{\frac{1}{2}}+\lambda_i(u)}{2}$, $i=1,\ldots,r.
$ 
In particular, one readily verifies that
$$
z\circ s=s\circ z=\rho\mu\,e,\quad \text{and}\quad \lambda_i(z)=\frac{\rho\mu}{\lambda_i(s)},\; i=1,\ldots,r.
$$
Using orthonormality of the Jordan frame, we obtain
\begin{align*}
\|s\|^2=\Big\|\sum_{i=1}^{r}\lambda_i(s)v_i\Big\|^2=\sum_{i=1}^{r}\lambda_i(s)^2\le U^2,\\
\|z\|^2=\Big\|\sum_{i=1}^{r}\lambda_i(z)v_i\Big\|^2=\sum_{i=1}^{r}\lambda_i(z)^2\le U^2.
\end{align*}
Invoking $\lambda_i(z)=\frac{\rho\mu}{\lambda_i(s)}$ yields
$
\|z\|^2=(\rho\mu)^2\sum_{i=1}^{r}\lambda_i(s)^{-2}\le U^2.
$
Define $
\lambda_{\max}(s):=\max_{1\le i\le r}\lambda_i(s)$  and $\lambda_{\min}(s):=\min_{1\le i\le r}\lambda_i(s).
$ Then, we have
\[
\lambda_{\max}(s)^2\le \sum_{i=1}^{r}\lambda_i(s)^2\le U^2,
\]
which implies
$
\lambda_{\max}(s)\le U.
$
Similarly, \[ \lambda_{\min}(s)^{-2}\le \sum_{i=1}^{r}\lambda_i(s)^{-2}\le \Big(\frac{U}{\rho\mu}\Big)^2, \] yielding $\lambda_{\min}(s)\ge \frac{\rho\mu}{U}$. Consequently, for $i=1,\ldots,r$,
\[
\frac{\rho\mu}{U}\le \lambda_i(s)\le U,
\qquad
\frac{\rho\mu}{U}\le \lambda_i(z)=\frac{\rho\mu}{\lambda_i(s)}\le U.
\]
Since $z$ and $z+s$ share the same Jordan frame, Proposition~\ref{Prop:eig value of Lyapunov} implies that $\L(z)$ and $(\L(z+s))^{-1}$ admit a common spectral decomposition. Thus, the operator $\L(z)(\L(z+s))^{-1}$ shares the same spectral decomposition, and its eigenvalues are given by
\[
    \left\{ \frac{\lambda_i(z)}{\lambda_i(z)+\lambda_i(s)} \,\big{|}\, 1 \le i \le r \right\} \; \text{and} \; \left\{ \frac{\lambda_i(z) + \lambda_j(z)}{\lambda_i(z) + \lambda_j(z)+\lambda_i(s) + \lambda_j(s)} \,\big{|}\, 1 \le i < j \le r \right\}.
\]
Consequently, $\lambda_{\min} \big(  \L(z) \L(z+s)^{-1} \big)\geq\dfrac{\rho\mu}{2U^2} \ge \dfrac{\varrho\mu}{2U^2}$, and
\begin{align*}
\lambda_{\min}(\nabla_{\lambda }^2 \eta)
&=\lambda_{\min}\!\Big(\mathcal{A}\L(z)(\L(z+s))^{-1}\mathcal{A}^{*}\Big) \\
&\ge \lambda_{\min}\!\big(\L(z) \L(z+s)^{-1}\big)\lambda_{\min}(\A\A^{*}) \\
&\ge \frac{\varrho\mu}{2U^2}\lambda_{\min}(\A\A^{*}).
\end{align*}
Together with $\lambda_{\max}(\nabla_{\lambda }^2 \eta)\le \lambda_{\max}(\A\A^{*})$, we obtain
\[
\cond(\nabla^2_{\lambda}\eta)
=\frac{\lambda_{\max}(\nabla^2_{\lambda}\eta)}{\lambda_{\min}(\nabla^2_{\lambda}\eta)}
\le
\frac{(2U^2)\,\lambda_{\max}(\A\A^{*})}{(\varrho\mu)\,\lambda_{\min}(\A\A^{*})}
=
\mathcal{O}\!\left(1/{\mu}\right).
\]
This completes the proof.
\end{proof}
\newpage
\section{Additional computational results}\label{Additional computational results}\quad
\begin{table}[H]
\scriptsize
\centering
\caption{\label{SDP-full-solve} Performance of the \NALM, SDPNAL+, SeDuMi, Clarabel and Hypatia on SDP problems (${\rm tol} = 10^{-6}$). ``a" means the solver fails to achieve the desired accuracy, despite the solution being near optimal.``m" means the procedure is out of memory. ``t" means the maximum time limit has been reached.}
\begin{tabular}{c|ccccc}
\toprule
Problem name   & NAL  & SDPNAL+ & SeDuMi & Clarabel  & Hypatia  \\
\hline
gpp100   & \textbf{0.391}  & a     & 0.716  & 356.521   & 11.340   \\
gpp124-1 & \textbf{0.392}  & 3.951   & 0.585  & 1499.430  & 6.561    \\
gpp124-2 & \textbf{0.324}  & 2.909   & 0.841  & 1449.016  & 10.835   \\
gpp124-3 & \textbf{0.373}  & 4.419   & 0.552  & 1355.056  & 7.839    \\
gpp124-4 & \textbf{0.303}  & 3.272   & 0.686  & 1483.746  & a   \\
gpp250-1 & \textbf{1.193}  & a     & 2.231  & m       & a  \\
gpp250-2 & \textbf{1.194}  & a     & 2.513  & m       & a   \\
gpp250-3 & \textbf{1.292}  & a     & 2.072  & m       & a   \\
gpp250-4 & \textbf{1.135}   & a     & 2.106  & m        & a   \\
gpp500-1 & \textbf{8.087} & 24.908  & 16.149 & m        & a \\
gpp500-2 & \textbf{7.972}  & 71.563  & 16.513 & m        & a  \\
gpp500-3 & \textbf{6.373}  & a     & 16.545 & m        & a  \\
gpp500-4 & \textbf{6.574}  & a     & 15.544 & m        & a  \\
mcp100   & 0.339  & 0.797   & \textbf{0.216}  & 3.706     & 8.401    \\
mcp124-1 & 0.390  & 1.227   & 0.282  & \textbf{0.059}     & 0.507    \\
mcp124-2 & \textbf{0.270}  & 1.222   & 0.300  & 1.861     & 0.603    \\
mcp124-3 & 0.375   & 1.015   & \textbf{0.294}  & 51.561    & 0.487    \\
mcp124-4 & \textbf{0.286}  & 0.863   & 0.399  & 252.866   & 0.552    \\
mcp250-1 & 1.315  & 3.743   & 1.104  & \textbf{0.394}     & 3.489    \\
mcp250-2 & \textbf{1.118}   & 2.960   & 1.128  & 76.725    & 3.983    \\
mcp250-3 & \textbf{1.006}  & 2.718   & 1.535  & 2201.879  & 3.792    \\
mcp250-4 & \textbf{0.887}  & 2.603   & 1.269  & 11211.088 & 3.536    \\
mcp500-1 & \textbf{5.725}  & 13.368  & 8.140  & 8.393     & 47.250   \\
mcp500-2 & 7.879  & 12.239  & \textbf{7.666}  & 3579.506  & 42.424   \\
mcp500-3 & \textbf{4.131}  & 12.624  & 7.806  & m        & 42.079   \\
mcp500-4 & \textbf{5.426}  & 13.565  & 7.682  & m        & 44.286   \\
maxG11   & 18.714          & 50.873  & 16.447 & \textbf{9.374}  & 253.576  \\
maxG32   & \textbf{196.612}   & 1055.894       & 326.272        & 1520.360        & Inf      \\
maxG51   & \textbf{26.491}    & 71.445         & 66.727         & m             & 773.879  \\
maxG55   & \textbf{4560.610}  & 10007.140      & 8051.370       & m             & m      \\
maxG60   & \textbf{14269.730} & 39346.690      & 23057.110      & m             & m      \\
qpG11    & 176.119            & 638.191        & 162.094        & \textbf{12.172} & m      \\
qpG51    & \textbf{650.667}   & 2014.699       & 667.986        & m             & m      \\
theta1   & 0.861              & 0.589          & \textbf{0.424} & 4.050           & 8.883    \\
theta2  & 3.011              & 0.923          & \textbf{0.461} & 195.800         & 1.525    \\
theta3   & 6.940              & \textbf{0.867} & 2.381          & 1858.640        & 11.011   \\
theta4   & 9.253              & \textbf{1.819} & 10.993         & 10660.689       & 48.023   \\
theta42 & 20.603 &  1.656 & 355.313 & t & 371.662\\
theta5    & 13.971             & \textbf{1.571} & 53.306         & m             & 168.493  \\
theta6    & 23.903             & \textbf{2.417} & 169.796        & m             & 453.377  \\
theta62 & 37.803 &  2.405 & 4070.368 & m & m\\
theta8 &  74.662 & 4.850 & 1089.775 & m & m\\
theta82 & 197.094 & 4.188 & t & m & m\\
\bottomrule
\end{tabular}
\end{table}
\newpage

\begin{table}[H]\scriptsize
\centering
\caption{\label{LP-full-solve-1} Performance of the \NALM, ABIP, SeDuMi, Clarabel and Hypatia on LP problems (${\rm tol} = 10^{-6}$). ``i" means the maximum iteration has been reached. ``n" means some numerical problem occur. ``d" means no sensible direction is found. ``a" means the solver fails to achieve the desired accuracy, despite the solution being near optimal. ``b" means the solver outputs primal or dual infeasibility. ``s" means only slow progress can be made. ``m" means the procedure is out of memory.}
\begin{tabular}{c|ccccc}
\toprule
Problem name     & NAL   & ABIP    & SeDuMi & Clarabel & Hypatia  \\
\hline
25fv47       & 0.266           & 9.892          & 3.689          & \textbf{0.235}  & 12.180         \\
80bau3b      & 0.500           & i              & n              & \textbf{0.422}  & 2857.200       \\
adlittle     & \textbf{0.009}  & 0.046          & 0.424          & 0.012           & 0.045          \\
afiro        & 0.007           & \textbf{0.004} & 0.043          & 0.008           & 0.020          \\
agg          & 0.073           & i              & 0.199          & \textbf{0.056}  & a              \\
agg2         & 0.114           & 1.246          & 0.103          & \textbf{0.049}  & 0.184          \\
agg3         & 0.112           & 1.314          & 0.147          & \textbf{0.048}  & b              \\
bandm        & 0.041           & 1.688          & 0.050          & \textbf{0.019}  & 0.098          \\
beaconfd     & 0.029           & 7.869          & 0.290          & \textbf{0.015}  & 0.042          \\
blend        & 0.011           & 0.010          & 0.024          & \textbf{0.008}  & 0.013          \\
bnl1         & \textbf{0.083}  & 68.533         & 0.191          & 0.116           & 2.676          \\
bnl2         & \textbf{0.371}  & 165.527        & 0.425          & 0.434           & 49.168         \\
boeing1      & 0.068           & 4.316          & 0.107          & \textbf{0.048}  & 0.414          \\
boeing2      & 0.027           & 8.517          & 0.081          & \textbf{0.017}  & 0.071          \\
bore3d       & 0.029           & i              & 0.049          & \textbf{0.019}  & 0.061          \\
brandy       & 0.035           & 0.233          & 0.044          & \textbf{0.023}  & 0.052          \\
capri        & 0.059           & 10.565         & 0.057          & \textbf{0.027}  & 0.150          \\
cycle        & 0.446           & 380.363        & 0.548          & \textbf{0.228}  & a              \\
czprob       & \textbf{0.107}  & 2.476          & 0.155          & 0.118           & 51.052         \\
d2q06c       & 0.671           & 83.108         & 1.086          & \textbf{0.563}  & 216.814        \\
d6cube       & 0.331           & 19.486         & d              & \textbf{0.258}  & 169.354        \\
degen2       & 0.113           & 0.212          & 0.051          & \textbf{0.030}  & 0.189          \\
degen3       & 0.792           & 2.492          & 0.296          & \textbf{0.207}  & 3.821          \\
dfl001       & 5.263           & 44.622         & 9.838          & \textbf{4.231}  & 978.806        \\
e226         & 0.031           & 0.700          & 0.039          & \textbf{0.018}  & 0.107          \\
etamacro     & 0.098           & i              & 0.073          & \textbf{0.052}  & 1.206          \\
fffff800     & 0.134           & i              & 0.172          & \textbf{0.100}  & 0.594          \\
finnis       & 0.049           & 1.690          & 0.117          & \textbf{0.044}  & 0.785          \\
fit1d        & 0.128           & 4.966          & 0.210          & \textbf{0.056}  & 3.157          \\
fit1p        & 0.829           & 8.070          & 0.310          & \textbf{0.032}  & 3.396          \\
fit2d        & 1.372           & 44.059         & 1.770          & \textbf{0.612}  & 2148.592       \\
fit2p        & 19.755          & 12.025         & 2.162          & \textbf{0.263}  & 2129.465       \\
forplan      & \textbf{0.031}  & i              & 0.090          & 0.051           & 2083.623       \\
ganges       & \textbf{0.072}  & 1.592          & n              & 0.073           & 9.576          \\
gfrd-pnc     & \textbf{0.027}  & 1.243          & 0.043          & 0.030           & 1.311          \\
greenbea     & 6.538           & i              & 1.709          & \textbf{0.715}  & 0.755          \\
greenbeb     & \textbf{0.676}  & i              & 3.311          & 0.971           & s              \\
grow15       & 0.051           & 0.471          & 0.063          & \textbf{0.023}  & 0.421          \\
grow22       & 0.070           & 3.332          & 0.075          & \textbf{0.033}  & 0.380          \\
grow7        & 0.027           & 0.257          & 0.031          & \textbf{0.014}  & 0.966          \\
israel       & 0.058           & 0.359          & 0.048          & \textbf{0.014}  & 0.093          \\
kb2          & \textbf{0.006}  & 0.176          & 0.029          & 0.008           & 0.061          \\
lotfi        & 0.018           & 1.182          & 0.044          & \textbf{0.015}  & 0.030          \\
maros        & 0.206           & i              & 0.227          & 0.105           & \textbf{0.077} \\
maros-r7     & 2.754           & 82.655         & \textbf{1.503} & 3.207           & b              \\
modszk1      & 0.058           & 8.393          & \textbf{0.050} & 0.051           & 2.316          \\
nesm         & 0.277           & 246.463        & 0.738          & \textbf{0.254}  & 2.212          \\
perold       & 0.173           & i              & 0.498          & \textbf{0.171}  & a              \\
pilot        & 1.672           & i              & 3.289          & \textbf{1.322}  & a              \\
pilot4       & 0.234           & i              & n              & \textbf{0.115}  & a              \\
pilot87      & \textbf{3.530}  & i              & 4.149          & 5.036           & 60.815         \\
pilot.ja     & 0.357           & i              & d              & \textbf{0.256}  & a              \\
pilot.we     & 0.208           & i              & 0.335          & \textbf{0.158}  & a              \\
pilotnov     & 0.356           & 204.911        & 0.551          & \textbf{0.150}  & 57.316         \\

\bottomrule

\end{tabular}
\end{table}
\begin{table}[H]\scriptsize
\centering
\caption{\label{LP-full-solve-2} Performance of the \NALM, ABIP, SeDuMi, Clarabel and Hypatia on LP problems (${\rm tol} = 10^{-6}$). ``i" means the maximum iteration has been reached. ``n" means some numerical problem occur. ``d" means no sensible direction is found. ``a" means the solver fails to achieve the desired accuracy, despite the solution being near optimal. ``b" means the solver outputs primal or dual infeasibility. ``s" means only slow progress can be made. ``m" means the procedure is out of memory.}
\begin{tabular}{c|ccccc}
		\toprule
Problem name     & NAL   & ABIP    & SeDuMi & Clarabel & Hypatia  \\
\hline
qap12        & 8.750           & 16.773         & \textbf{7.208} & 21.799          & a              \\
qap15        & \textbf{38.556} & 88.810         & 86.459         & 190.510         & 213.678        \\
qap8         & 0.557           & 0.324          & \textbf{0.139} & 0.271           & a              \\
recipe       & \textbf{0.007}  & 0.015          & 0.033          & 0.008           & 0.130          \\
sc105        & \textbf{0.005}  & 0.006          & 0.017          & 0.007           & 0.017          \\
sc205        & \textbf{0.008}  & 0.044          & 0.021          & 0.010           & 0.042          \\
sc50a        & \textbf{0.004}  & \textbf{0.004} & 0.013          & 0.006           & 0.012          \\
sc50b        & \textbf{0.003}  & 0.004          & 0.010          & 0.005           & 0.009          \\
scagr25      & \textbf{0.014}  & 0.096          & 0.031          & \textbf{0.014}  & 0.158          \\
scagr7       & \textbf{0.009}  & 0.025          & 0.021          & \textbf{0.009}  & 0.026          \\
scfxm1       & 0.037           & 0.188          & 0.048          & \textbf{0.031}  & 0.211          \\
scfxm2       & 0.059           & 0.597          & 0.093          & \textbf{0.051}  & 0.873          \\
scfxm3       & 0.089           & 0.879          & 0.108          & \textbf{0.045}  & 2.317          \\
scorpion     & 0.013           & 0.171          & 0.022          & \textbf{0.012}  & 0.076          \\
scrs8        & 0.043           & 21.748         & 0.063          & \textbf{0.036}  & 1.238          \\
scsd1        & \textbf{0.008}  & 0.020          & 0.017          & 0.010           & 0.332          \\
scsd6        & \textbf{0.012}  & 0.101          & 0.020          & 0.019           & 1.308          \\
scsd8        & 0.028           & 0.078          & \textbf{0.025} & \textbf{0.025}  & 9.278          \\
sctap1       & \textbf{0.018}  & 0.139          & 0.030          & 0.022           & 0.268          \\
sctap2       & 0.043           & 0.208          & 0.052          & \textbf{0.028}  & 5.269          \\
sctap3       & 0.058           & 0.166          & 0.055          & \textbf{0.039}  & 11.327         \\
seba         & 0.193           & 1.229          & 0.336          & \textbf{0.039}  & 1.280          \\
share1b      & \textbf{0.015}  & 4.584          & 0.039          & 0.016           & 0.072          \\
share2b      & 0.010           & 0.255          & 0.019          & \textbf{0.008}  & 0.023          \\
shell        & 0.045           & 5.655          & 0.096          & \textbf{0.034}  & b              \\
ship04l      & 0.033           & 0.226          & 0.041          & \textbf{0.031}  & 6.236          \\
ship04s      & 0.027           & 0.341          & 0.037          & \textbf{0.021}  & 1.610          \\
ship08l      & \textbf{0.067}  & 0.544          & 0.073          & \textbf{0.067}  & 47.923         \\
ship08s      & 0.039           & 0.295          & 0.041          & \textbf{0.030}  & 7.425          \\
ship12l      & \textbf{0.080}  & 5.382          & 0.103          & 0.241           & 147.448        \\
ship12s      & \textbf{0.041}  & 0.842          & 0.051          & 0.048           & 13.051         \\
sierra       & 0.128           & 128.871        & 0.173          & \textbf{0.069}  & 18.847         \\
stair        & 0.076           & i              & 0.055          & \textbf{0.034}  & 0.167          \\
standata     & 0.028           & 0.170          & 0.030          & \textbf{0.014}  & 1.474          \\
standgub     & 0.029           & 0.077          & 0.030          & \textbf{0.014}  & 1.896          \\
standmps     & 0.035           & 0.280          & 0.059          & \textbf{0.025}  & 1.318          \\
stocfor1     & \textbf{0.009}  & 0.214          & 0.021          & \textbf{0.009}  & 0.018          \\
stocfor2     & 0.117           & 22.956         & 0.190          & \textbf{0.054}  & 7.042          \\
stocfor3     & 1.622           & i              & 1.743          & \textbf{0.670}  & m              \\
truss        & 0.297           & 14.589         & 0.166          & \textbf{0.154}  & 436.221        \\
tuff         & 0.059           & 0.585          & 0.049          & \textbf{0.026}  & 0.384          \\
vtp.base     & \textbf{0.022}  & i              & 0.076          & \textbf{0.022}  & 0.137          \\
wood1p       & \textbf{0.151}  & 2.565          & \textbf{0.151} & 0.154           & 13.285         \\
woodw        & 0.457           & 26.188         & 0.312          & \textbf{0.207}  & 715.555        \\
cre-a        & 0.379           & 105.363        & \textbf{0.282} & b               & 120.912        \\
cre-b        & \textbf{3.611}  & a              & 4.187          & b               & m              \\
cre-c        & 0.230           & 8.419          & 0.290          & \textbf{0.109}  & 104.901        \\
cre-d        & 2.309           & 44.103         & 1.734          & \textbf{1.176}  & m              \\
ken-07       & 0.099           & 20.281         & 0.095          & \textbf{0.051}  & 21.314         \\
ken-11       & 0.503           & 581.546        & 2.038          & \textbf{0.477}  & m              \\
ken-13       & 1.889           & 1942.202       & 2.550          & \textbf{0.137}  & m              \\
ken-18       & 20.317          & a              & 36.695         & \textbf{0.866}  & m              \\
osa-07       & 1.374           & 16.748         & 1.687          & \textbf{0.836}  & m              \\
osa-14       & 3.565           & 618.997        & 5.439          & \textbf{2.177}  & m              \\
osa-30       & 6.258           & 3668.879       & 16.572         & \textbf{4.234}  & m              \\
osa-60       & 29.761          & i              & 69.033         & \textbf{12.937} & m              \\
pds-02       & 0.219           & 1.590          & 0.170          & \textbf{0.087}  & s              \\
pds-06       & 2.947           & 13.926         & 3.613          & \textbf{1.373}  & m              \\
pds-10       & 7.911           & 35.151         & 12.390         & \textbf{7.449}  & m              \\
pds-20       & \textbf{22.461} & 112.854        & 82.746         & 54.618          & m \\   
\bottomrule
\end{tabular}
\end{table}
\begin{table}[H]\scriptsize
\centering
\caption{\label{SOCP-full-solve} Performance of the \NALM, ABIP, SeDuMi, Clarabel and Hypatia on SOCP problems (${\rm tol} = 10^{-6}$). ``m" means the procedure is out of memory. ``t" means the maximum time limit has been reached.}
\begin{tabular}{c|ccccc}
\toprule
Problem name   & NAL  & ABIP & SeDuMi & Clarabel  & Hypatia  \\
\hline
meb\_100\_10       & \textbf{0.238}  & 0.553    & 1.942    & 0.287    & 9.873    \\
meb\_100\_50       & 0.112  & 3.220    & 0.089    & \textbf{0.048}    & 53.904   \\
meb\_200\_20       & 0.154  & 5.205    & 0.060    & \textbf{0.037}    & 33.811   \\
meb\_200\_40       & 0.160  & 10.648   & \textbf{0.061}    & 0.089    & 211.177  \\
meb\_100\_100      & 0.119  & 5.181    & 0.118    & \textbf{0.090}    & 350.715  \\
meb\_1000\_400     & \textbf{2.008}  & 2118.366 & 4.040    & 13.036   & m \\
meb\_1000\_800     & \textbf{3.293}  & t & 9.903    & 45.316   & m \\
meb\_1000\_1200    & \textbf{4.921}  & t & 19.711   & 82.332   & m \\
meb\_1000\_1600    & \textbf{6.375}  & t & 28.609   & 128.107  & m \\
meb\_1000\_2000    & \textbf{6.387}  & t & 320.866  & 186.744  & m \\
meb\_3000\_300     & \textbf{3.323}  & t & 5.832    & 72.307   & m \\
meb\_6000\_300     & \textbf{8.760}  & t & 11.774   & 273.846  & m \\
meb\_12000\_300    & \textbf{18.518} & t & 26.854   & 1053.071 & m \\
meb\_2000\_1000    & \textbf{7.764}  & t & 17.159   & 179.728  & m \\
meb\_2000\_2000    & \textbf{15.990} & t & 92.210   & 497.009  & m \\
lasso\_800\_10     & 0.342  & 0.716    & 0.315    & \textbf{0.312}    & 10.618   \\
lasso\_1000\_400   & 0.577  & 0.377    & 3.612    & \textbf{0.354}    & 2.524    \\
lasso\_1000\_800   & 0.355  & 0.201    & \textbf{0.130}    & 0.150    & 4.719    \\
lasso\_2000\_1200  & \textbf{1.276}  & 2.130    & 50.099   & 2.063    & 28.017   \\
lasso\_3000\_1600  & \textbf{3.138}  & 4.996    & 154.311  & 13.347   & 77.003   \\
lasso\_4000\_2000  & \textbf{2.091}  & 6.174    & 9.344    & 37.829   & 193.377  \\
lasso\_4000\_3000  & 3.010  & 4.725    & \textbf{1.838}    & 2.333    & 338.969  \\
lasso\_5000\_3000  & \textbf{4.504}  & 19.614   & 10.750   & 19.245   & 467.144  \\
lasso\_6000\_3000  & \textbf{4.210}  & 17.883   & 38.485   & 119.126  & 625.230  \\
lasso\_8000\_3000  & \textbf{8.446}  & 25.827   & 854.469  & 86.801   & 1130.628 \\
lasso\_9000\_3000  & \textbf{10.937} & 30.700   & 2040.184 & 95.892   & m \\
lasso\_10000\_4000 & \textbf{12.275} & 57.519   & 3757.698 & 191.409  & m \\
lasso\_10000\_3000 & \textbf{13.859} & 29.482   & 2303.860 & 95.627   & m \\
\bottomrule
\end{tabular}
\end{table}

\section*{Acknowledgments}
Rui-Jin Zhang and Ruoyu Diao contributed equally and are joint first authors.

\bibliographystyle{myamsplain}
\bibliography{references}


\end{document}